\documentclass[a4paper,10pt]{article}
\usepackage{geometry}
	\usepackage[utf8]{inputenc}
	\usepackage{amssymb}
	\usepackage{amsmath}
	\usepackage{amsfonts}
	\usepackage{amsthm}
	\usepackage{stmaryrd}
	\usepackage{mathrsfs}
	\usepackage{wrapfig}
	\usepackage{algorithm, algorithmic}

	\usepackage{amsopn}
	\usepackage{mathtools}
	\usepackage{pgfplots}
	\usepackage{pgfplotstable}
	\usepackage{booktabs}
	\usepackage{cite}
\usepackage[hidelinks,colorlinks=true,allcolors=blue]{hyperref}
	
	\usepackage{tikz}
	\usetikzlibrary{matrix,decorations.pathreplacing,calc}
	
	\pgfkeys{tikz/mymatrixenv/.style={decoration=brace,every left delimiter/.style={xshift=3pt},every right delimiter/.style={xshift=-3pt}}}
	\pgfkeys{tikz/mymatrix/.style={matrix of math nodes,left delimiter=[,right delimiter={]},inner sep=2pt,column sep=1em,row sep=0.5em,nodes={inner sep=0pt}}}
	\pgfkeys{tikz/mymatrixbrace/.style={decorate}}
	
	\newcommand\mymatrixbraceoffsetv{0.2em}

	\newcommand*\mymatrixbracebottom[4][m]{
	    \draw[mymatrixbrace] ($(#1.south west)!(#1-1-#3.south east)!(#1.south east)-(0,\mymatrixbraceoffsetv)$)
	        -- node[below=2pt] {#4} 
	        ($(#1.south west)!(#1-1-#2.south west)!(#1.south east)-(0,\mymatrixbraceoffsetv)$);
	}
	\catcode`@=11
	\font\tenex=cmex10 
	\newdimen\p@renwd
	\setbox0=\hbox{\tenex B} \p@renwd=\wd0 
	\def\bmat#1{\begingroup \m@th
	   \setbox\z@\vbox{\def\cr{\crcr\noalign{\kern2\p@\global\let\cr\endline}}%
	     \ialign{$##$\hfil\kern2\p@\kern\p@renwd&\thinspace\hfil$##$\hfil
	       &&\quad\hfil$##$\hfil\crcr
	       \omit\strut\hfil\crcr\noalign{\kern-\baselineskip}%
	       #1\crcr\omit\strut\cr}}%
	   \setbox\tw@\vbox{\unvcopy\z@\global\setbox\@ne\lastbox}%
	   \setbox\tw@\hbox{\unhbox\@ne\unskip\global\setbox\@ne\lastbox}%
	   \setbox\tw@\hbox{$\kern\wd\@ne\kern-\p@renwd\left[\kern-\wd\@ne
	     \global\setbox\@ne\vbox{\box\@ne\kern2\p@}%
	     \vcenter{\kern-\ht\@ne\unvbox\z@\kern-\baselineskip}\,\right]$}%
	   \null\;\vbox{\kern\ht\@ne\box\tw@}\endgroup}
	\catcode`@=12
	
	\usepackage{natbib}
	\setcitestyle{square}
	\setcitestyle{numbers}
	\setcitestyle{comma}
    \usepackage{caption}
\usepackage{subcaption}
	
	\newcommand{\R}{\mathbb {R}}
	\newcommand{\C}{\mathbb {C}}
	\newcommand{\trace}{\mathrm{tr}}
	\newcommand{\diag}{\mathrm{diag}}
	\newtheorem{theorem}{Theorem}
	\newtheorem{corollary}[theorem]{Corollary}
	\newtheorem{lemma}[theorem]{Lemma}
	\newtheorem{remark}[theorem]{Remark}
	\newtheorem{definition}[theorem]{Definition}
	\newtheorem{proposition}[theorem]{Proposition}

	\title{Divide and conquer methods for functions of matrices with banded or hierarchical low-rank structure}
	
	\author{Alice Cortinovis\footnote{MATH-ANCHP, École Polytechnique Fédérale de Lausanne, Station 8, 1015 Lausanne, Switzerland. E-mail: alice.cortinovis@epfl.ch. The work of Alice Cortinovis has been supported by the SNSF research project \emph{Fast algorithms from low-rank updates}, grant number: 200020\_178806.} \and Daniel Kressner\footnote{MATH-ANCHP, École Polytechnique Fédérale de Lausanne, Station 8, 1015 Lausanne, Switzerland. E-mail: daniel.kressner@epfl.ch} \and Stefano Massei\footnote{Centre for Analysis, Scientific Computing and Applications (CASA), TU Eindhoven, Eindhoven,
Netherlands. E-mail: s.massei@tue.nl}}
	
	\date{}
	
\begin{document}
	
	\maketitle
	
	\begin{abstract}
	    This work is concerned with approximating matrix functions for banded matrices, hierarchically semiseparable matrices, and related structures. We develop a new divide-and-conquer method based on (rational) Krylov subspace methods for performing low-rank updates of matrix functions.
Our convergence analysis of the newly proposed method proceeds by establishing   relations to best polynomial and rational approximation. When only the trace or the diagonal of the matrix function is of interest, we demonstrate -- in practice and in theory -- that convergence can be faster. For the special case of a banded matrix, we show that the divide-and-conquer method reduces to a much simpler algorithm, which proceeds by computing matrix functions of small submatrices. Numerical experiments confirm the effectiveness of the newly developed algorithms for computing large-scale matrix functions from a wide variety of applications.
	\end{abstract}
	
	

\section{Introduction}

The task of evaluating matrix functions $f(A)$ for $A\in\mathbb R^{n\times n}$, such as the matrix exponential or the matrix square root, has been studied intensively in the last two decades \cite{higham2008functions}. These problems arise in the numerical solution of partial differential equations~\cite{Druskin2016,Lee2010}, electronic structure calculations~\cite{bekas2007estimator, Goedecker1999}, and social network analysis~\cite{Estrada2010}, as we will see in more detail in Section~\ref{sec:testDC} and Section~\ref{sec:examplepoly}. In this paper we are concerned with  leveraging sparse and low-rank structures in $f(A)$ that are inherited from $A$. More specifically, we consider the case where $A$ is banded, or has off-diagonal blocks with low-rank, e.g. when $A$ is hierarchically semiseparable (HSS)~\cite{Xia2010}, and we design efficient procedures for computing and storing $f(A)$ and related quantities of interest.  

When $A$ is banded and $f$ is well approximated by a low-degree polynomial on the spectrum of $A$, the matrix function $f(A)$ can usually be well approximated by a banded matrix. Many a priori results confirm this property.
For example, the entries of the inverse of a tridiagonal matrix $A$ decay quickly with the distance to the diagonal, provided that $A$ is well conditioned~\cite{Demko1984}. 
Such decay properties extend to inverses of symmetric banded matrices~\cite{Demko1984}, to more general matrix functions of symmetric banded matrices~\cite{Benzi1999}, and to symmetric sparse matrices with more general sparsity patterns~\cite{Benzi2015}. When $A$ is not symmetric, one can proceed via diagonalization assuming a well conditioned eigenvector matrix~\cite{BenziRazouk2007, Pozza2019} or by leveraging the Crouzeix-Palencia result~\cite{CrouzeixPalencia2017}, at the price of considering polynomial approximations of $f(z)$ on the numerical range of $A$~\cite{BenziBoito2014}. 

For computing a matrix function $f(A)$ when $A$ is banded, to directly exploit the approximate bandedness of $f(A)$ one can use an a priori polynomial approximation $p$ and evaluate $f(A) \approx p(A)$. Compared to $A$, the width of the band gets multiplied by the degree of the polynomial $p$. This technique is used, for instance, in electronic structure methods~\cite{Goedecker1999} combined with Chebyshev interpolation~\cite{Benzi2013}. In~\cite{BenziRazouk2007}, polynomial approximation is combined with a dropping strategy in order to maintain a low bandwidth in the approximation of $f(A)$. A possible alternative to a priori polynomial approximations is to adapt an existing method for dense matrices to banded matrices, possibly combining it with thresholding in order to maintain sparsity; for example, Newton-Schultz iterations have been used for the sign function in the context of electronic structure calculations~\cite{Dawson2018, Nemeth2000}. For functions of banded Toeplitz matrices, structured thresholding techniques have been designed in order to maintain a Toeplitz plus low-rank structure \cite{Bini2016,bini2019exponential}. 

When $f$ cannot be well approximated on the spectrum of $A$ by a low-degree polynomial, the techniques described above may lead to poor results. Often, low-rank structures come to the rescue. To illustrate this, let us consider again an invertible tridiagonal matrix $A$. When $A$ is very ill-conditioned, the decay of the off-diagonal entries of $A^{-1}$ mentioned above vanishes; however, all the off-diagonal blocks of $A^{-1}$ have rank 1. Therefore, $A^{-1}$ can be efficiently represented by a hierarchically semiseparable matrix~\cite[Section 3.9]{Hackbusch2015}. 
This also means that if we can choose a priori a rational function $r$ with small degree that well approximates $f$, then we can approximate $f(A) \approx r(A)$ in the HSS format. An advantage of this approach is that it also works for matrices with off-diagonal low-rank structure. For the exponential, there exists an excellent rational approximation on the negative real axis~\cite{Andersson1981}, which implies a good approximation of $\exp(-A)$ for a symmetric positive definite (SPD) matrix $A$ even when the norm of $A$ is large; see~\cite{Guettel2013} for further examples. Another favorable class of functions is the one of Markov functions, that has been recently discussed in~\cite{Beckermann2021} in the context of a Toeplitz matrix argument. Rational functions approximating $f$ can also be obtained by discretizing the Cauchy integral representation of the function; this approach is used, for instance, for the exponential operator~\cite{Gavrilyuk2002}, for step functions arising in the computation of spectral projectors~\cite{KressnerSusnjara2017}, and for matrix functions that may have singularities inside the contour of integration~\cite{MasseiRobol2017}. 
An alternative to a priori rational approximation is the use of iterative methods, such as for the matrix sign function~\cite{Grasedyck2003} or the matrix square root of a symmetric positive definite matrix~\cite[Section 15.3]{Hackbusch2015}; the iterations can be done in HSS arithmetic and possibly some truncation strategies are needed in order to maintain a low-rank structure.

In this work, we design new algorithms for approximating matrix functions of matrices which can be recursively decomposed as the sum of a block diagonal matrix $D$ and a low-rank correction. This is the case for banded matrices, HSS matrices, and sparse matrices corresponding to graphs with community structure~\cite{Newman2006}. As shown in~\cite{BeckermannKressnerSchweitzer2018, BCKS2020}, the matrix function update $f(A) - f(D)$ is often numerically low-rank and can be efficiently approximated using a subspace projection approach with suitable Krylov subspaces. In this work we perform the evaluation of $f(D)$ recursively, leading to a divide-and-conquer (D\&C) algorithm. Similarly to the a priori bounds on $f(A)$ mentioned above, we prove that the effectiveness of the D\&C algorithm is related to best polynomial or rational approximation. 
However, let us emphasize that the use of Krylov subspaces bypasses the need of choosing an a priori polynomial or rational approximation to $f$ and this can be beneficial if there are some outliers in the spectrum of $A$.

For banded matrices $A$, polynomial Krylov subspaces associated to low-rank updates inherit sparsity. Thanks to this fact, we can use a splitting approach to develop a method that allows for a more compact description of the low-rank updates and a more efficient implementation.
Our algorithm is based on covering $A$ with overlappping blocks and only needs the evaluation of $f$ on these blocks. A related, although significantly different, technique has been proposed in~\cite{Shao2014} for approximating the exponential of infinite banded matrices.  
The equivalence of our method with low-rank updates allows us to prove a convergence result that connects the error of the algorithm with polynomial approximations of $f$. 

In many applications, only specific quantities associated to $f(A)$ are needed. For example, the trace of matrix functions is used to compute spectral densities~\cite{Lin2016}, log determinants~\cite{Gardner2018}, Schatten $p$-norms~\cite{Dudley2020}, the Estrada index of a graph~\cite{Estrada2010}, and also arises in lattice quantum chromodynamics~\cite{Wu2016}. The diagonal of a matrix function is needed, for instance, in Density Functional Theory~\cite{bekas2007estimator}, electronic structure calculations~\cite{Lin2009}, and uncertainty quantification~\cite{Tang2012}. 
Our algorithms can be simplified in case one is only interested in such quantities. We observe accelerated convergence and we confirm this by theoretical results.

The remainder of the paper is organized as follows. In Section~\ref{sec:lowrankupdates} we recall some results on low-rank updates of matrix functions and we present a new convergence result regarding the update of the trace of a matrix function.  Section~\ref{sec:general} is dedicated to the D\&C algorithm for matrix functions and its convergence analysis. Numerical experiments for banded and HSS matrix arguments are presented in Section~\ref{sec:testDC}. In Section~\ref{sec:splitting} we present and analyze a block diagonal splitting algorithm that is specialized for banded matrices. The performances of the splitting algorithm are validated in Section~\ref{sec:examplepoly}. Finally, some conclusions are drawn in Section~\ref{sec:conclusions}.

\section{Low-rank updates of matrix functions}\label{sec:lowrankupdates}

In this section, we summarize the algorithms from~\cite{BeckermannKressnerSchweitzer2018,BCKS2020} on low-rank updates of matrix functions and improve their convergence analysis in the case of trace approximation.

Given $A,R\in\mathbb R^{n\times n}$, with $R$ of low rank, and a function $f(z)$ defined on the spectra of $A$ and $A+R$, one aims at computing the update
$$
f(A+R)-f(A).
$$
It turns out that this matrix usually has low numerical rank, in the sense that it can be well approximated by a low-rank matrix.
This motivates to search for approximations of the form
\begin{equation}\label{eq:update}
f(A+R)-f(A)  \approx U_m X_m(f) V_m^T,
\end{equation}
where $U_m,V_m$ are orthonormal bases of (low-dimensional) subspaces $\mathcal{U}_m, \mathcal{V}_m$ of $\R^n$.
%
%
For reasons explained in~\cite{BeckermannKressnerSchweitzer2018,BCKS2020}, a suitable choice for the (small) coefficient matrix $X_m(f)$ is the $(1, 2)$-block of the matrix
\begin{equation*}
f\left ( \begin{bmatrix} U_m^T A U_m & U_m^T R V_m \\ 0 & V_m^T (A+R) V_m \end{bmatrix} \right ) =: f\left ( \begin{bmatrix} G_m & U_m^T R V_m \\ 0 & H_m \end{bmatrix} \right ) .
\end{equation*}

The quality of the approximation~\eqref{eq:update} strongly depends on the choice of $\mathcal U_m$, $\mathcal V_m$. A natural choice are (rational) Krylov subspaces: Given a factorization of the low-rank matrix $R$, \[R =  B J C^T, \quad B,C \in \R^{n\times \text{rank}(R)}, \quad J \in \R^{\text{rank}(R) \times \text{rank}(R)},  \]
we let $\mathcal U_m$ and $\mathcal V_m$ be Krylov subspaces generated with the matrices $A$ and $A^T$ and starting (block) vectors $B$ and $C$, respectively. When choosing a \emph{polynomial Krylov subspace} then \begin{equation*}
\mathcal U_m = \mathcal{K}_m(A, B) := \mathrm{span} \big[ B, A B, A^2 B, \ldots, A^{m-1} B\big]
\end{equation*}
and $\mathcal V_m$ is defined analogously. When choosing a \emph{rational Krylov subspace} \cite{ruhe1984rational} associated with $q(z) = (z - \xi_1) \cdots (z - \xi_m)$ for prescribed poles $\mathbf{\xi}=(\xi_1, \ldots, \xi_m)^T \in \C^m$, then
\begin{equation} \label{eq:ratkrylov}
\mathcal U_m = \mathcal{RK}_m(A, B,\mathbf \xi) := \mathrm{span}\big[q_m(A)^{-1} B, q_m(A)^{-1} AB, q_m(A)^{-1} A^2B, \ldots, q_m(A)^{-1} A^{m-1}B\big].
\end{equation}
To make sure that $\mathcal U_m$ is real, the set of poles is assumed to be closed under complex conjugation. Also, we allow for infinite poles and consider the polynomial Krylov subspace as the particular case where $\xi_j=\infty$, $j=1,\dots,m$. 

The orthonormal bases $U_m, V_m$ of 
$\mathcal{RK}_m(A, U_R,\mathbf \xi)$, $\mathcal{RK}_m(A^T, V_R,\mathbf \xi)$ are computed with the block rational Arnoldi method described in~\cite{elsworth2020block,berljafa2014rational}. This computation is performed incrementally with respect to $m$ and yields the compressed matrices $U_m^TAU_m$ and $V_m^T(A+R)V_m$ nearly for free. For choosing $m$, we use the heuristic error estimate from~\cite{BeckermannKressnerSchweitzer2018}:
\begin{eqnarray*}
\|f(A+R)-f(A)-U_{m-d} X_{m-d}(f)V_{m-d}^T\|_F & \approx & \|U_m X_m(f)V_m^T-U_{m-d} X_{m-d}(f)V_{m-d}^T\|_F \nonumber \\
&=& \|X_m(f)-\left[\begin{smallmatrix}
	X_{m-d}(f)&0\\ 0&0
	\end{smallmatrix}\right]\|_F
\end{eqnarray*}
 for a small integer $d$, the so called \emph{lag parameter}; $\| \cdot \|_F$ denotes the Frobenius norm of a matrix.
 The whole procedure is summarized in Algorithm~\ref{alg:kryl}.
 Each step of the block rational Arnoldi method in lines \ref{line:k1}-\ref{line:k2} requires either matrix-vector products with $A,A^T$ (for an infinite pole) or solving shifted linear systems with $A,A^T$ (for a finite pole). When only a few different finite poles are present, it can be beneficial to precompute the LU factorization of the shifted matrix $A$ and reuse it across several steps.
We refer to~\cite[Section 3.1]{BCKS2020} and the references therein concerning further implementation details.
 

\begin{algorithm}[H]
	\caption{Krylov subspace projection for approximating $f(A+BJC^T)-f(A)$}
	\label{alg:kryl}
	\begin{algorithmic}[1]
		\PRINT \textsc{Krylov\_proj}($A,B,J,C, \mathbf{\xi}, f(z),d,\varepsilon$) \hspace{2cm} $\triangleright\ \  \mathbf{\xi}=(\xi_1,\dots,\xi_{m_{\max}})^T$
			\FOR{$m=1,\dots,m_{\max}$}
			\STATE $\mathbf{\xi}^{(m)} \gets(\xi_1,\dots,\xi_{m})^T$ 
			\STATE{Compute orthonormal basis $U_m$ of $\mathcal{RK}_{m}(A,B,\mathbf{\xi}^{(m)})$ and $G_m=U_m^TAU_m$}\label{line:k1}
			\STATE{Compute orthonormal basis $V_m$ of $\mathcal{RK}_{m}(A^T,C,\mathbf{\xi}^{(m)})$ and $H_m=V_m^T(A+BJC^T)V_m$}\label{line:k2}
			\STATE{\label{line:compxmf}Compute $X_m(f)$ as the $(1,2)$ block of $f\left(\left[\begin{smallmatrix}
				G_m&(U_m^T B)J(C^T V_m)\\ & H_m
				\end{smallmatrix}\right]\right)$}\label{line:fun-proj}
			\IF{$m>d$ and $\left \|X_m(f)-\left[\begin{smallmatrix}
				X_{m-d}(f)&0\\ 0&0
				\end{smallmatrix}\right]\right \|_F<\varepsilon$}
			\STATE\textbf{break}
			\ENDIF
			\ENDFOR
			\RETURN $U_m,X_m(f),V_m$
		\end{algorithmic}
	\end{algorithm}
	
When $A$ and $R$ are symmetric, we can choose $C = B$. It follows that $U_m = V_m$ and hence only one basis needs to be generated; line~\ref{line:k2} of Algorithm~\ref{alg:kryl} is skipped. Moreover, the computation of $X_m(f)$ in line~\ref{line:compxmf} simplifies to 
\begin{equation*}
X_m(f) = f(U_m^T (A+R) U_m) - f(U_m^T A U_m).
\end{equation*}

\subsection{Exactness results and convergence of Algorithm~\ref{alg:kryl}}\label{sec:traceSymmetric}

We let $\Pi_m$ denote the space of polynomials with degree bounded by $m$. In~\cite[Theorem 3.2]{BeckermannKressnerSchweitzer2018} it is shown that, when using polynomial Krylov subspaces, the approximation $U_m X_m(f) V_m^T$ returned by Algorithm~\ref{alg:kryl} equals the exact update $f(A+R)-f(A)$
when $f \in \Pi_m$. In~\cite[Theorem 3.3]{BCKS2020} this was extended to the rational Krylov subspace~\eqref{eq:ratkrylov}; in this case $U_m X_m(f) V_m^T$ is exact for all $f \in \Pi_m / q_m$, that is, all rational functions of the form $p(z) / q_m(z)$ with $p \in \Pi_m$. Such exactness results are turned into convergence results via polynomial/rational approximation.

\begin{remark}\label{rmk:exactness}
For future reference, we note that the exactness results explained above 
also hold when $U_m$ and $V_m$ are orthonormal bases of subspaces of $\R^n$ which \emph{contain} the Krylov subspaces $\mathcal{U}_m$ and $\mathcal{V}_m$, respectively.
\end{remark}

When $A$ and $R$ are symmetric, a better exactness result holds when considering the update of the trace, that is, the approximation
\begin{equation*} 
\trace(f(A + BJB^T) - f(A)) \approx \trace(U_m X_m(f) U_m^T) = \trace(X_m(f))
\end{equation*}
instead of the approximation of the full update.

 \begin{theorem}\label{thm:exactnessTrace}
  Let $A \in \R^{n \times n}$ and $J \in \R^{b \times b}$ be symmetric, and let $B \in \R^{n \times b}$. Let $U_m$ be an orthonormal basis of $\mathcal{K}_m(A, B)$. Then 
  \begin{equation*}
   \trace(X_m(p)) = \trace(p(A + BJB^T) - p(A)) \text{ for all } p \in \Pi_{2m}.
  \end{equation*}
 \end{theorem}
 \begin{proof}
 By linearity it is sufficient to show that the theorem holds for monomials, that is, we need to prove that
 \begin{equation*}
  \trace\left ((U_m^T(A + BJB^T)U_m)^j \right ) - \trace\left ((U_m^T A U_m)^j \right ) = \trace \left ((A + BJB^T)^j \right ) - \trace(A^j)
 \end{equation*}
 for $j = 0, 1, 2, \ldots, 2m$. The left hand side is a sum of terms of the following form:
 \begin{equation}\label{eq:1term}
  \trace \left ( (U_m^T A U_m)^{a_0} (U_m^T BJB^T U_m)^{b_1} (U_m^T A U_m)^{a_1}\cdots  (U_m^T BJB^T U_m)^{b_h} (U_m^T A U_m)^{a_h} \right ),
 \end{equation}
for some $h \ge 1$, $a_0, a_h \ge 0$, $a_1, \ldots, a_{h-1} \ge 1$, $b_1, \ldots, b_h \ge 1$, and $a_0 + b_1 + \ldots + a_{h-1} + b_h + a_h = j$. By~\cite[Lemma 3.1]{Saad1992} we have that
\begin{equation}\label{eq:exactpolyK}
U_m (U_m^T A U_m)^k U_m^T B = A^k B
\end{equation}
for all $k = 0, \ldots, m-1$. Moreover, it is easy to see that for $k \ge 1$ we have $(U_m^T BJB^T U_m)^k = U_m^T (BJB^T )^kU_m=U_m^T B (JB^T B)^{k-1} J B^T U_m$. Then, using \eqref{eq:exactpolyK} and the cyclic property of the trace we rewrite~\eqref{eq:1term} as
\begin{equation}\label{eq:shuffled}
\begin{split}
&\trace\left((U_m^T A U_m)^{a_0} (U_m^T BJB^T U_m)^{b_1} (U_m^T A U_m)^{a_1}\cdots  (U_m^TBJB^TU_m)^{b_h} (U_m^T A U_m)^{a_h}\right)\\
&=\trace\left((U_m^T A U_m)^{a_0}U_m^TB\left(\prod_{i=1}^{h-1} C_{a_i,b_i}\right)(JB^TB)^{b_h-1} J B^T U_m(U_m^T A U_m)^{a_h}\right)\\
&=\trace\left(C_{a_0+a_h,b_h}\prod_{i=1}^{h-1} C_{a_i,b_i}\right)
\end{split}
\end{equation}
with $C_{a,b}:= (JB^TB)^{b-1} J B^T U_m(U_m^T A U_m)^{a}U_m^TB$ for $b \ge 1$ and $0 \le a \le 2m-1$. We claim that $C_{a,b}=(JB^TB)^{b-1}JB^TA^a B$: If $a \le m-1$, this follows directly from the exactness property~\eqref{eq:exactpolyK}; if $a \ge m$, we write $C_{a,b}$ as 
\begin{equation*}
 (JB^TB)^{b-1} J \underbrace{B^T U_m (U_m^T A U_m)^{m-1} U_m^T}_{B^TA^{m-1}} A \underbrace{U_m (U_m^T A U_m)^{a-m} U_m^T B}_{A^{a-m}B}
\end{equation*}
and use the exactness property~\eqref{eq:exactpolyK} on the two selected parts to arrive at the same conclusion. Finally, incorporating the rightmost factor $B$ of $C_{a_i,b_i}$ into $C_{a_{i+1},b_{i+1}}$  we obtain that \eqref{eq:shuffled} is equal to
\begin{equation*}
 \trace \left ( (JB^TB)^{b_h-1}JB^TA^{a_0+a_h}U_m^T (BJB^T)^{b_1} A^{a_1} \cdots (BJB^T)^{b_h} A^{a_{h-1}} B\right ).
\end{equation*}
By means of the cyclic property of the trace we finally get $$\trace \left ( A^{a_0} (BJB^T)^{b_1} A^{a_1} \cdots (BJB^T)^{b_h} A^{a_h}\right )$$
which matches the terms in the expansion of $\trace\left ((A + BJB^T)^j \right ) - \trace(A^j)$.
 \end{proof}
 
 For a set $\mathbb{E}$ and a function $f$ we denote $\|f\|_{\mathbb{E}} := \sup_{z \in \mathbb{E}} |f(z)|$. Moreover, we  indicate with $\Lambda(A)$ the convex hull of the eigenvalues of $A$. 
The following theorem provides an a priori estimate on the error of the approximation of the trace of a matrix function update obtained by Algorithm~\ref{alg:kryl}.
 
 \begin{theorem}\label{thm:convTraceUpdate}
  Let $A$ be symmetric and let $f$ be defined on an interval $\mathbb{E} \subset \R$ containing the eigenvalues of $A$ and  $A + BJB^T$. Then
  \begin{equation*}
   | \trace(f(A + BJB^T) - f(A)) - \trace(X_m(f)) | \le 4n \min_{p \in \Pi_{2m}} \| f - p \|_{\mathbb{E}}.
  \end{equation*}
 \end{theorem}
 \begin{proof}
 By Theorem~\ref{thm:exactnessTrace}, for all polynomials $p \in \Pi_{2m}$ we have that 
 \begin{align*}
  |\trace(f(A + BJB^T) & - f(A)) - \trace(X_m(f)) |  \\
  & =  |\trace((f- p)(A + BJB^T)) - \trace((f-p)(A)) \\
  & \qquad + \trace((f-p)(U_m^T (A + BJB^T) U_m)) -\trace((f-p)(U_m^T A U_m))  | \\
  &  \le |\trace((f-p)(A + BJB^T))| + |\trace((f- p)(A))|  \\
  & \qquad + |\trace((f-p)(U_m^T (A + BJB^T) U_m))| + |\trace((f-p)(U_m^T A U_m))| \\
  & \le n\|(f-p)(A + BJB^T)\|_2 + n\|(f- p)(A)\|_2 \\
  & \qquad + n\|(f-p)(U_m^T (A + BJB^T) U_m)\|_2 + n\|(f-p)(U_m^T A U_m)\|_2.
 \end{align*}
 For a normal matrix $X$ and a function $g$, it holds that $\|g(X)\|_2 \le \|g\|_{\Lambda(X)}$, where $\| \cdot \|_2$ denotes the spectral norm of a matrix. As the spectral intervals of all matrices $A+BJB^T$, $A$, $U_m^T(A + BJB^T)U_m$, and $U_m^T A U_m$ are all contained in $\mathbb{E}$, it follows that the right hand side of the above equation is upper bounded by $4n\|f-p\|_{\mathbb{E}}$. Taking the minimum over all polynomials $p \in \Pi_{2m}$ concludes the proof.
 \end{proof}
 
 For example, consider SPD matrices $A\in\R^{n \times n}$ and $J \in \R^{b \times b}$, a matrix $B \in \R^{n \times b}$, and denote by $[\alpha, \beta]$ an interval containing the eigenvalues of $A$ and $A + BJB^T$. The best polynomial approximation error on such interval when $f$ is the square root function is proportional to $\gamma^m$ for $\gamma := (\sqrt{\beta/\alpha}-1)/(\sqrt{\beta/\alpha}+1)$; see, e.g.,~\cite[Theorem 8.2]{Trefethen2013}. Therefore, the error in the approximation of $f(A+BJB^T) - f(A)$ via Algorithm~\ref{alg:kryl} decreases geometrically with rate $\gamma$, while the error in the approximation of $\trace(f(A+BJB^T)-f(A))$ decreases with rate $\gamma^2$ thanks to Theorem~\ref{thm:convTraceUpdate}, that is, twice as fast.
 
 \paragraph{Numerical examples.}
 Figure~\ref{fig:doublespeedsymmetric} reports numerical experiments to explore the scope of the result of Theorem~\ref{thm:convTraceUpdate}. For this purpose, we have applied Algorithm~\ref{alg:kryl} with polynomial Krylov subspaces to random symmetric and nonsymmetric matrices $A$. In Figure~\ref{fig:doublespeedsymmetric} (a) and (b), the double speed of convergence predicted by Theorem~\ref{thm:convTraceUpdate} is only observed for the trace and when $A,R$ are symmetric. In all other situations, when approximating the diagonal or when $A$ is nonsymmetric, there is no significant difference in the convergence. In Figure~\ref{fig:doublespeedsymmetric} (c) a rational Krylov subspace method is used and the double speed of convergence of the trace approximation error disappears even when $A,R$ are symmetric.
 

\begin{figure}[ht]
     \centering
     \begin{subfigure}[b]{0.3\textwidth}
         \centering
         \includegraphics[width=\textwidth]{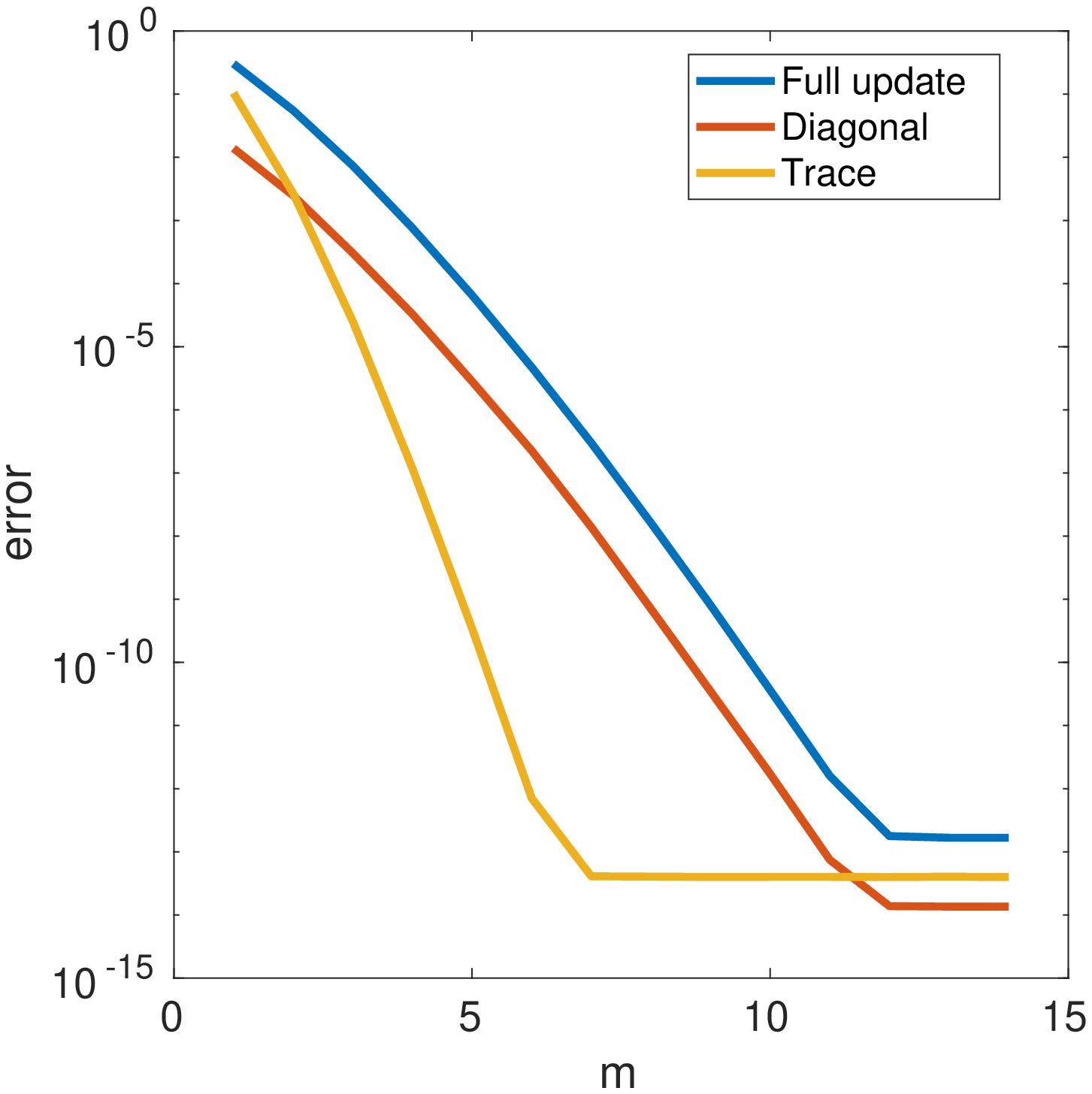}
         \caption{Polynomial Krylov applied to symmetric $A$ and $R$}
     \end{subfigure}
     \hfill
     \begin{subfigure}[b]{0.3\textwidth}
         \centering
         \includegraphics[width=\textwidth]{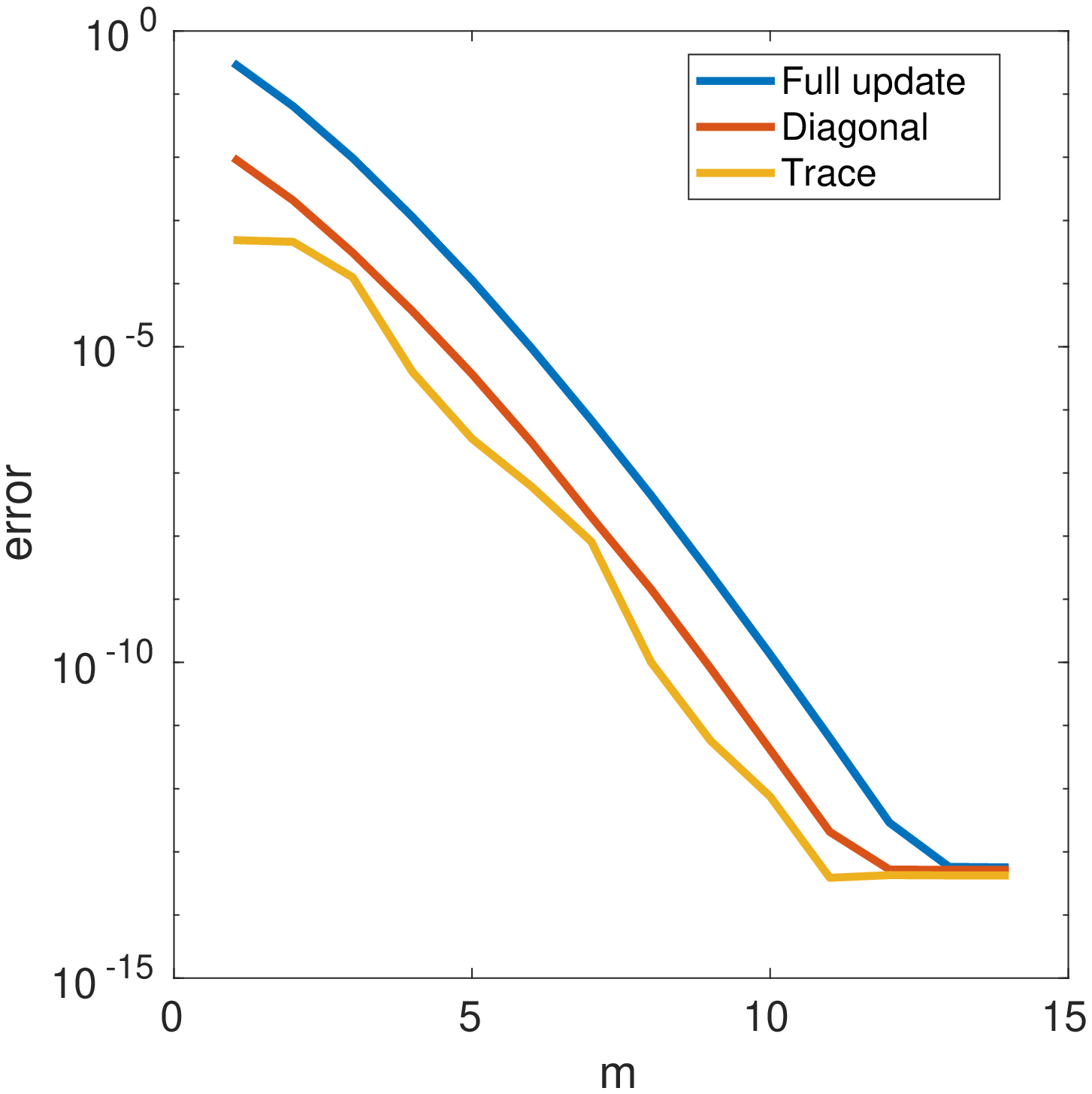}
         \caption{Polynomial Krylov applied to non-symmetric $A$}
     \end{subfigure}
     \hfill
     \begin{subfigure}[b]{0.3\textwidth}
         \centering
         \includegraphics[width=\textwidth]{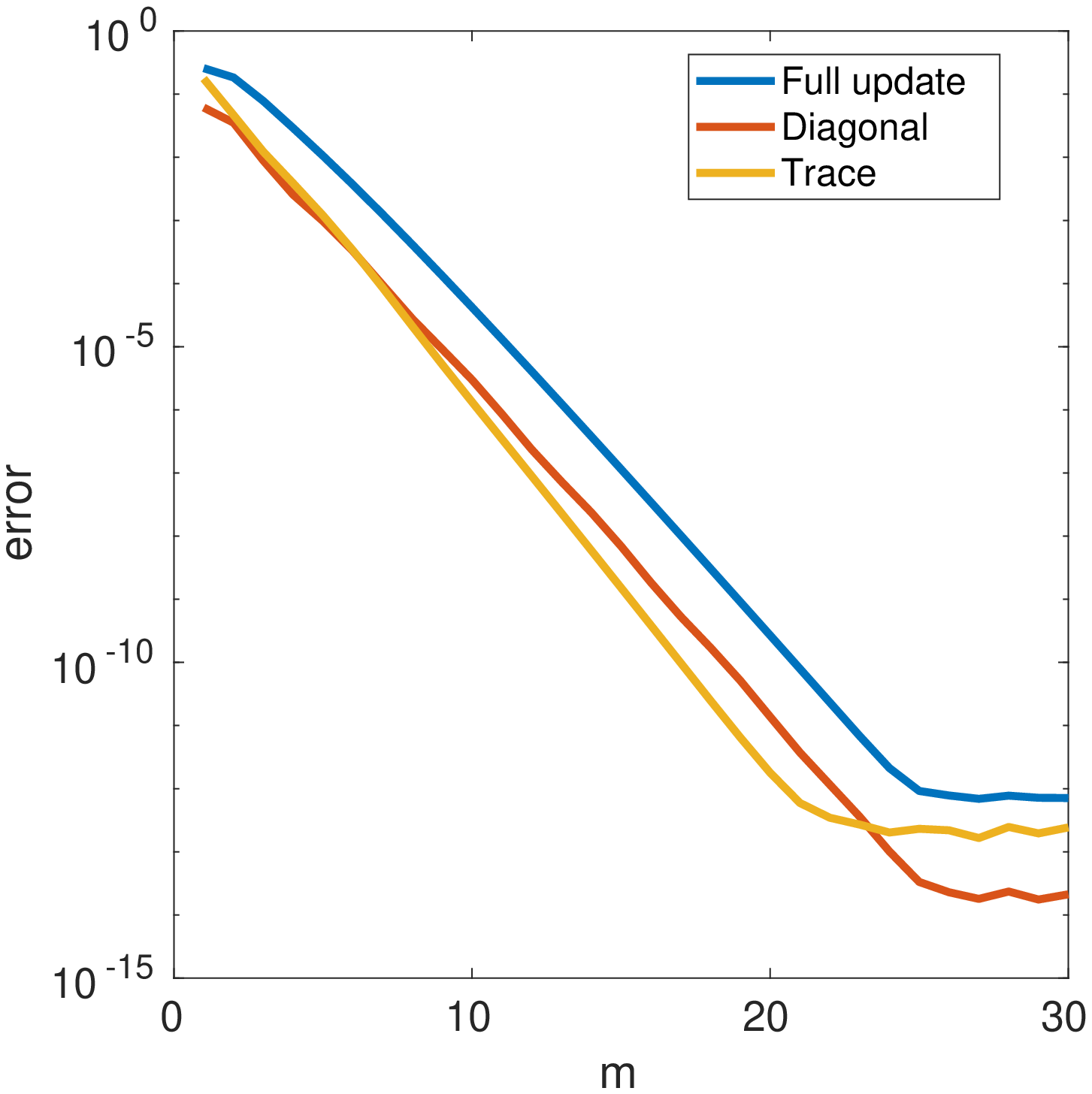}
         \caption{Rational Krylov applied to symmetric $A$ and $R$}
     \end{subfigure}
        \caption{Convergence of the errors $\|f(A+R) - f(A) - U_m X_m(f) V_m^T\|_F$,
 $\|\text{diag}(f(A+R) - f(A) - U_m X_m(f) V_m^T )\|_2$, and 
 $|\trace(f(A+R) - f(A)) - U_m X_m(f) V_m^T )|$ for $f=\exp$.}
 \label{fig:doublespeedsymmetric}
\end{figure}

\section{Divide-and-conquer for matrix functions}\label{sec:general}
\subsection{Divide-and-conquer for matrices with low-rank off-diagonal blocks}
In this section we use low-rank updates to devise a new divide-and-conquer (D\&C) method for functions of matrices that have low-rank off-diagonal blocks. More specifically, let us assume that $A$ can be block partitioned as \begin{equation}\label{eq:AdAo}
A=
\underbrace{\begin{bmatrix}
A_{11}&\\
&A_{22}
\end{bmatrix}}_{A_D}+\underbrace{\begin{bmatrix}
&A_{12}\\
A_{21}
\end{bmatrix}}_{A_O},\quad A_{11}\in\mathbb R^{n_1\times n_1},\quad A_{22}\in\mathbb R^{n_2\times n_2},
\end{equation}
where the off-diagonal part $A_O$ has low rank and the diagonal blocks can be recursively block partitioned in the same fashion. Examples of matrix structures that have this property are banded matrices and \emph{hierarchically semiseparable (HSS) matrices}~\cite{chandrasekaran2006fast}; see also Section~\ref{sec:hss} below. 

The computation of $f(A)$ is split in two tasks: computing $f(A_D)$ and computing  $f(A)-f(A_D)$. The latter quantity is approximated via Algorithm~\ref{alg:kryl} exploiting that $A_O=A-A_D$ has low rank; the former  decouples into the computation of $f(A_{11})$ and $f(A_{22})$. Since we assume that the blocks $A_{ii}$ can again be decomposed into the form~\eqref{eq:AdAo}, the described procedure is applied recursively for computing $f(A_{ii})$, $i=1,2$. Finally, when the size of a block $A_{ii}$ is below a minimal block size $n_i\leq n_{\min}$, we evaluate $f(A_{ii})$ with a standard dense method, like the scaling and squaring method~\cite{higham2008functions} for $f=\exp$.



Algorithm~\ref{alg:D&C} summarizes the described D\&C method for matrix functions.
\begin{algorithm}[ht]
	\caption{Template of D\&C algorithm for matrix functions}
	\label{alg:D&C}
	\begin{algorithmic}[1]
 		\PRINT \textsc{D\&C\_funm}($A, \mathbf{\xi}, f(z),d,\varepsilon,n_{\min}$, \texttt{flag}) \hspace{3.5cm} $A\in\mathbb R^{n\times n}$
		\IF{$n\leq n_{\min}$}
		\IF{\texttt{flag} $=$ ``full"}
		\RETURN $f(A)$\label{line:full}
				\ELSIF{\texttt{flag} $=$ ``diagonal"}
				\RETURN \texttt{diag}($f(A)$)
				\ELSIF{\texttt{flag} $=$ ``trace"}
				\STATE Compute the eigenvalues $\lambda_j$ $j=1,\dots,n$, of $A$
				\RETURN $\sum_{j=1}^mf(\lambda_j)$
				\ENDIF
		\ENDIF
		\STATE Given a decomposition~\eqref{eq:AdAo}, retrieve a low-rank factorization $A_O=B J C^T$\label{step:fact}
		\STATE $[U,X,V]\gets $\textsc{Krylov\_proj}($A_D,B,J, C,\xi,f(z),d,\varepsilon$) \hspace{3.7cm}\emph{ (Algorithm~\ref{alg:kryl})} \label{line:kryl}
		\STATE $F_{11}\gets$  \textsc{D\&C\_funm}($A_{11}, \mathbf{\xi}, f(z),d,\varepsilon,n_{\min}$, \texttt{flag}) \hspace{4.4cm} \emph{(Recursion)}\label{line:rec1}
		\STATE $F_{22}\gets$  \textsc{D\&C\_funm}($A_{22}, \mathbf{\xi}, f(z),d,\varepsilon,n_{\min}$, \texttt{flag})\hspace{4.5cm} \emph{(Recursion)}\label{line:rec2}
\IF{\texttt{flag} $=$ ``full"}
\RETURN $\left[\begin{smallmatrix}
F_{11}\\ &F_{22}
\end{smallmatrix}\right]+UXV^T$\label{line:sum}
\ELSIF{\texttt{flag} $=$ ``diagonal"}
\RETURN \label{line:diag} $\left[\begin{smallmatrix}
	F_{11}\\ F_{22}
\end{smallmatrix}\right]+$ \texttt{diag}($UXV^T$)
\ELSIF{\texttt{flag} $=$ ``trace"}
\RETURN \label{line:trace} $F_{11}+F_{22}\ +$ \texttt{trace}($V^TUX$)
\ENDIF
	\end{algorithmic}
\end{algorithm}
The D\&C method simplifies when certain selected quantities of $f(A)$, like the diagonal or the trace, are of interest. Because of linearity, it suffices to evaluate the diagonal or the trace of the low-rank update $UXV^T\approx f(A)-f(A_D)$; see lines~\ref{line:diag} and~\ref{line:trace} of Algorithm~\ref{alg:D&C}.


The exactness properties of Algorithm~\ref{alg:kryl} discussed in Section~\ref{sec:traceSymmetric} directly imply the following result.
\begin{proposition}\label{prop:exactDC}
Let $A \in \R^{n \times n}$ and consider $q_m(z) := \prod_{i=1}^m (z-\xi_i)$ for a set of $m$ poles $\xi_1, \ldots, \xi_m \in \C \cup \{\infty\}$ closed under complex conjugation. Then Algorithm~\ref{alg:D&C} applied to $A$ and a function $f \in \Pi_m / q_m$ is exact, provided that Algorithm~\ref{alg:kryl} called in Line~\ref{line:kryl} utilizes all $m$ poles.
\end{proposition}

\subsection{Algorithm~\ref{alg:D&C} for banded matrices}

Let us first consider the application of Algorithm~\ref{alg:D&C} to a banded matrix $A$ with bandwidth $b$, that is, $a_{ij}=0$ whenever $|i-j|>b$. Then  the off-diagonal part $A_O$ in the decomposition~\eqref{eq:AdAo} has rank at most $2b$.  Under the idealistic assumption that Algorithm~\ref{alg:kryl} converges in a constant number of iterations (independent of $n$), computing the low-rank update on the top level of recursion requires $\mathcal O(b^2n)$ operations when using  either polynomial or rational Krylov subspaces. 
Thus, the total complexity of Algorithm~\ref{alg:D&C} is  $\mathcal O(b^2 n \log n)$, provided that $n_{\min} = \mathcal O(1)$.

\begin{remark}\label{rmk:spdrankb}
By an appropriate correction of the diagonal blocks in the decomposition~\eqref{eq:AdAo}, it is possible to reduce the rank of the off-diagonal part to $b$. Although this clearly has the potential to result in lower-dimensional  Krylov subspaces in the low-rank update, it also bears the danger of leading to diagonal blocks for which $f$ is not defined or difficult to approximate. When $A$ is SPD then the rank-$b$ update can be chosen such that the diagonal blocks remain SPD~\cite[Section 4.4.2]{KressnerMasseiRobol2019}.
\end{remark}

\begin{remark}\label{rmk:bandedpolynomial}
When Algorithm~\ref{alg:D&C} is used with polynomial Krylov subspaces for banded $A$ then it can be shown that the output is again banded (but with larger bandwidth). However, in such a situation a much simpler approach is possible, which will be described in Section~\ref{sec:splitting}.
\end{remark}

\subsection{Storing the output of Algorithm~\ref{alg:D&C} using HSS matrices}\label{sec:hss}

Except for the situation described in Remark~\ref{rmk:bandedpolynomial}, 
the approximation of $f(A)$ constructed in line~\ref{line:sum} of Algorithm~\ref{alg:D&C} is not banded. To still efficiently represent this approximation, we use HSS matrices. In the following we give a brief introduction to HSS matrices; see~\cite{massei2020hm,Xia2010} 
for more details.

We start by formalizing the concept of recursive partitioning.
\begin{definition} \label{def:clustertree}
	Given $n\in \mathbb N$, let $\mathcal T_L$ be a perfect binary tree of depth $L$ whose nodes are
	subsets of $\{ 1, \ldots, n \}$. We say that $\mathcal T_L$ is
	a \emph{cluster tree} if it satisfies:
	\begin{itemize}
		\item The root is $I^{0}_1 := I = \{ 1, \ldots, n \}$.
		\item The nodes at level $\ell$, denoted by $I^{\ell}_1, \ldots, I^{\ell}_{2^\ell}$, form
		a partitioning of $\{1, \ldots, n\}$ into consecutive indices:
		\[
		I_i^\ell = \{ n^{(\ell)}_{i-1}+1\ldots, n^{(\ell)}_{i}-1, n_i^{(\ell)} \}
		\]
		for some integers $0 = n^{(\ell)}_{0} \le  n^{(\ell)}_{1} \le \cdots \le n^{(\ell)}_{2^\ell} = n$, $\ell = 0,\ldots , L$. In particular, if $n_{i-1}^{(\ell)}=n_i^{(\ell)}$ then $I_i^{\ell}=\emptyset$.
		\item The children form a partitioning of their parent. 
	\end{itemize}
\end{definition}
Usually, the cluster tree $\mathcal T_L$ is  defined such that  the index sets on the same level $\ell$ have nearly equal cardinalities and the depth of the tree is determined by a minimal diagonal block
size $n_{\mathrm{min}}$ for stopping the recursion. In particular, if $n = 2^L n_{\mathrm{min}}$, such a construction yields a
perfectly balanced binary tree of depth $L$.

The structure of an HSS matrix is determined by $\mathcal T_L$. For 
any siblings $I^\ell_i,I^\ell_j$ in $\mathcal T_L$, the corresponding off-diagonal block of $A$ is denoted by $A(I^\ell_i,I^\ell_j)$. For an HSS matrix of HSS rank $k$, every such off-diagonal block has rank at most $k$ and thus admits a factorization
\[
A(I^\ell_i,I^\ell_j) = U_i^{(\ell)} S_{i,j}^{(\ell)} (V_j^{(\ell)})^T, \quad S_{i,j}^{(\ell)} \in\mathbb R^{k\times k}, \quad U_i^{(\ell)}\in \mathbb R^{n_i^{(\ell)} \times k},\quad V_j^{(\ell)} \in \mathbb R^{n_j^{(\ell)} \times k}.
\]
Moreover, the factors $U_i^{(\ell)}, V_j^{(\ell)}$ are nested across different levels of $\mathcal T_L$ \cite{Xia2010}. More specifically, there exist so called \emph{translation operators}, $R_{U,i}^{(\ell)}, R_{V,j}^{(\ell)}\in\mathbb R^{2k\times k}$ such that
\begin{equation*}
U_i^{(\ell)} = \begin{bmatrix} U_{2i-1}^{(\ell+1)} & 0 \\ 0 & U_{2i}^{(\ell+1)} \end{bmatrix} R_{U,i}^{(\ell)}, \qquad
V_j^{(\ell)} = \begin{bmatrix} V_{2j-1}^{(\ell+1)} & 0 \\ 0 & V_{2j}^{(\ell+1)} \end{bmatrix} R_{V,j}^{(\ell)}, 
\end{equation*}
where $I_{2i-1}^{\ell+1}, I_{2i}^{\ell+1}$  and $I_{2j-1}^{\ell+1}, I_{2j}^{\ell+1}$ denote the children of $I_{i}^{\ell}$  and $I_{j}^{\ell}$, respectively. Given the bases $U_i^{(L)}$ and $V_i^{(L)}$ at the deepest level $L$, the low-rank factors $U_i^{(\ell)}$ and $V_i^{(\ell)}$ for the higher levels $\ell=1,\ldots,L-1$, can be retrieved by means of the  translation operators.
 Therefore, the representation of $A$  only requires to store: the diagonal blocks $D_i:=A(I^L_i,I^L_i)$, the bases $U_i^{(L)}$, $V_i^{(L)}$, the core factors $S_{i,j}^{(\ell)}$, $S_{j,i}^{(\ell)}$ and the translation operators $R_{U,i}^{(\ell)}$, $R_{V,i}^{(\ell)}$.  Therefore, the storage cost is $\mathcal O(kn)$. Note that we have used a uniform rank $k$ for the off-diagonal blocks to simplify the description; in practice these ranks are chosen adaptively. 
 
 In the context of Algorithm~\ref{alg:D&C}, we choose a cluster tree that aligns with  the (recursive) decompositions~\eqref{eq:AdAo}. In turn, the sum at line~\ref{line:sum} is performed using HSS arithmetic and is combined  with a re-compression step to mitigate the increase of the HSS rank. This costs $\mathcal O(k^2n)$ operations, assuming that the HSS ranks of $F_{11},F_{22}$ and the rank of $UXV^T$ are $\mathcal O(k)$~\cite{massei2020hm}. 
 
\subsection{Algorithm~\ref{alg:D&C} for HSS matrices}\label{sec:hss2}

We now discuss the situation when the HSS structure is not only used for storing the output of Algorithm~\ref{alg:D&C} but when the input matrix $A$ itself is also an HSS matrix. 
In this case the decomposition~\eqref{eq:AdAo} is aligned with the cluster tree $\mathcal T_L$ associated with $A$ as this choice guarantees that the rank of $A_O$ is bounded by $2k$ and that the outcome inherits the same cluster tree of the input matrix.
In addition, fast algorithms for matrix operations are available within the HSS format~\cite{massei2020hm}. More specifically, complexity $\mathcal O(kn)$ is achieved for the matrix-vector product; solving  linear systems and computing the corresponding matrix factorizations cost $\mathcal O(k^2n)$. Algorithm~\ref{alg:D&C} leverages these features as follows:
\begin{itemize}
	\item Retrieve the low-rank factorization at line~\ref{step:fact} by means of the translation operators ($\mathcal O(k^2n)$).
	\item Generate the Krylov subspaces in \textsc{Krylov\_proj} by  performing matrix-vector products and/or solving shifted linear systems with HSS algorithms.
	\item Use the HSS structures of $A_{11},A_{22}$ in the recursive calls at lines~\ref{line:rec1}-\ref{line:rec2} and return HSS matrices $F_{11}$ and $F_{22}$.
	
\end{itemize}
Let us analyze the cost of Algorithm~\ref{alg:D&C} for the input $(\mathcal T_L,k)$-HSS matrix $A$, with $L=\mathcal O(\log(n))$, and \texttt{flag} equals ``full''.
We again make the idealistic assumption that \textsc{Krylov\_proj} converges in a constant number of iterations, independent of $k$ and $n$, and that the outcome of the (compressed) sum at line~\ref{line:sum} has always HSS rank $\mathcal O(k)$. Then, we have that the low-rank updates at level $\ell\in\{0,1,\dots,L-1\}$ cost $\mathcal O(k^2(n_i^{(\ell)}-n_{i-1}^{(\ell)}))$, $i=1,\dots,2^\ell$, when using either polynomial or rational Krylov subspaces. Since the sum at line~\ref{line:sum} costs $\mathcal O(k^2(n_i^{(\ell)}-n_{i-1}^{(\ell)}))$ too, the  asymptotic complexity of each non base level of the recursion is $\mathcal O(k^2\sum_{i=1}^{2^\ell}(n_i^{(\ell)}-n_{i-1}^{(\ell)}))=\mathcal O(k^2n)$. The base of the recursion requires to evaluate $\mathcal O(n/n_{min})$ functions of matrices of size at most $n_{\min}\times n_{\min}$; assuming a cubic cost for matrix function evaluations yields $\mathcal O(n_{\min}^2n)$. Hence, the overall complexity of Algorithm~\ref{alg:D&C} is $\mathcal O(k^2n\log(n))$.  
	\subsection{Convergence results for D\&C algorithm}

Convergence results for Algorithm~\ref{alg:D&C} can be obtained from the convergence results on low-rank updates of matrix functions discussed in Section~\ref{sec:traceSymmetric}.
In the following, we let $\mathcal{T}_L$ denote the (perfect binary) tree of depth $L$ associated with the recursive decompositions performed in line~\ref{step:fact}.
\begin{theorem}\label{thm:convDC}
 Let $A$ be symmetric and let $f$ be a function analytic on an interval $\mathbb{E}$ containing the eigenvalues of $A$. Suppose that Algorithm~\ref{alg:D&C}  uses  rational Krylov subspaces with poles $\xi_1, \ldots, \xi_m$, closed under complex conjugation, for computing updates.  
 Then the output $F_A$ of Algorithm~\ref{alg:D&C} satisfies
 \begin{equation*}
  \| f(A) - F_A \|_2 \le 4 L \cdot \min_{r \in \Pi_m / q_m} \| f - r \|_{\mathbb{E}},
 \end{equation*}
 where $q_m(z) = \prod_{i=1}^m (z - \xi_i)$.
\end{theorem}

\begin{proof}
 Using the index sets contained in $\mathcal{T}_L$ (see Definition~\ref{def:clustertree}), the matrices to which Algorithm~\ref{alg:D&C} is applied to in the $\ell$th level of recursion are denoted by $A_j^{\ell} := A(I_j^{\ell},I_j^{\ell})$ for $\ell < L$. Analogously, we let $G_j^{\ell}$ denote the update of the form $UXV^T$ computed in line~\ref{line:kryl}. 
 We aim at proving the following bound for the error of Algorithm~\ref{alg:D&C}:
 \begin{equation}\label{eq:levels}
  \|f(A) - F_A\|_2 \le \sum_{\ell = 0}^{L-1} \max_{j=1,\ldots,2^{\ell}} \left \| f(A^{\ell}_j) - \begin{bmatrix}f(A^{\ell+1}_{2j-1}) & \\ & f(A_{2j}^{\ell+1}) \end{bmatrix} - G^{\ell}_j \right \|_2.
 \end{equation}
 This bound implies the statement of the theorem because by~\cite[Theorem 4.5]{BCKS2020} each term appearing in the sum can be bounded by
\begin{equation*}
\left \| f(A^{\ell}_j) - \begin{bmatrix}f(A_{2j-1}^{\ell+1}) & \\ & f(A_{2j}^{\ell+1}) \end{bmatrix} - G^{\ell}_j \right \|_2 \le 4 \min_{r \in \Pi_m / q_m} \|f - r\|_{\mathbb{E}},
\end{equation*}
where we used that the eigenvalues of principal submatrices of $A$ are contained in $\mathbb{E}$.

The proof of~\eqref{eq:levels} is by induction on $L$, the number of levels. When $L = 1$, the definition of $F_A$ yields
\begin{equation*}
    \|f(A) - F_A \|_2 = \left \| f(A_1^0) - \begin{bmatrix} f(A_1^1) & \\ & f(A_2^1) \end{bmatrix} - G_1^0 \right \|_2.
\end{equation*}
Now, suppose that~\eqref{eq:levels} holds for $L-1$. Then the result for $L \ge 2$ is proven by observing 
\begin{align*}
 \| f(A) - F_A \|_2 & = \left \lVert f(A_1^0) - \left ( \begin{bmatrix} F_{A^{1}_{1}} & \\ & F_{A^{1}_{2}} \end{bmatrix} + G_1^0 \right ) \right \rVert_2 \\
 & = \left \lVert f(A_1^0) - \begin{bmatrix} f(A^{1}_{1}) & \\ & f(A^{1}_{2}) \end{bmatrix} - G_1^0 + \begin{bmatrix} f(A^{1}_{1}) & \\ & f(A^{1}_{2}) \end{bmatrix} - \begin{bmatrix} F_{A^{1}_{1}} & \\ & F_{A^{1}_{2}} \end{bmatrix} \right \rVert_2\\
 & \le \left \lVert f(A_1^0) - \begin{bmatrix} f(A^{1}_{1}) & \\ & f(A^{1}_{2}) \end{bmatrix} - G_1^0 \right \rVert_2 + \left \lVert\begin{bmatrix} f(A^{1}_{1}) - F_{A^{1}_{1}} & \\ & f(A^{1}_{2}) - F_{A^{1}_{2}} \end{bmatrix} \right \rVert_2\\
 & = \left \lVert f(A_1^0) - \begin{bmatrix} f(A^{1}_{1}) & \\ & f(A^{1}_{2}) \end{bmatrix} - G_1^0 \right \rVert_2 + \max_{k \in \{1,2\}} \| f(A^{1}_{k}) - F_{A^{1}_{k}}\|_2.
\end{align*}
Each of the terms $\|f(A^{1}_{k}) - F_{A^{1}_{k}}\|_2$ corresponds to applying Algorithm~\ref{alg:D&C} with a cluster tree of depth $L-1$, for which~\eqref{eq:levels} holds by the induction assumption; therefore,~\eqref{eq:levels} also holds for $L$. 
\end{proof}

\begin{corollary}
 Under the assumptions of Theorem~\ref{thm:convDC}, when using polynomial Krylov subspaces in Algorithm~\ref{alg:kryl}, we have that
 \begin{equation*}
 |\mathrm{trace}(f(A)) - \mathrm{trace}(F_A)| \le 4nL \min_{p \in \Pi_{2m}} \|f - p \|_{\mathbb{E}}.
 \end{equation*}
\end{corollary}
\begin{proof}
 Analogously to the proof of Theorem~\ref{thm:convDC}, we can bound
 \begin{equation*}
  |\mathrm{trace}(f(A)) - \mathrm{trace}(F_A)| \le \sum_{\ell = 0}^{L-1} \sum_{j=1}^{2^{\ell}} \left | \mathrm{trace}(f(A_j^{\ell})) - \mathrm{trace}\begin{bmatrix}f(A_{2j-1}^{\ell}) & \\ & f(A_{2j}^{\ell}) \end{bmatrix} - \mathrm{trace}(G_j^{\ell}) \right |
 \end{equation*}
and use Theorem~\ref{thm:convTraceUpdate} to conclude.
\end{proof}

	\section{Numerical tests for Algorithm~\ref{alg:D&C}}\label{sec:testDC}

In this section we test Algorithm~\ref{alg:D&C} on a variety of matrices and functions coming from different applications.  The minimum block size parameter $n_{\min}$ is set to $256$ for all our experiments, and the tolerance is $\varepsilon = 10^{-8}$ for all experiments, unless otherwise noted. The lag parameter in Algorithm~\ref{alg:kryl} is set to $d = 1$. The algorithm has been implemented in Matlab, version 9.9 (R2020b), and all numerical experiments in this work have been run on an eight-core Intel Core i7-8650U 1.90 GHz CPU, with 256 KB of level 2 Cache and 16 GB of RAM. The code for reproducing the experiments in this section and in Section~\ref{sec:examplepoly} is available at~\url{https://github.com/Alice94/MatrixFunctions-Banded-HSS}. The computations with HSS matrices have been performed using the hm-toolbox~\cite{massei2020hm}. This requires choosing a minimum block size and a tolerance parameter, which we set to be equal to $n_{\min}$ and $\varepsilon$, respectively. 

In all tables referring to the computation of matrix functions $f(A)$ the columns denoted by ``Err" contain the relative error in the Frobenius norm computed -- whenever possible -- with respect to the value of $f(A)$ obtained by dense arithmetic.

\subsection{Space-fractional diffusion equation without source}\label{sec:fractional}
Let us consider the fractional diffusion problem:
\[
\begin{cases}\frac{\partial u(x,t)}{\partial t}=\frac{\partial^\alpha u(x,t)}{\partial_- x^\alpha}+\frac{\partial^\alpha u(x,t)}{\partial_+ x^\alpha}&(x,t)\in (0,1)\times (0,T]\\
u(x,t)=0&(x,t)\in (\mathbb R\setminus[0,1]) \times [0,T]\\
u(x,0) = u_0(x)&x\in[0,1]
\end{cases}
\]
where $\alpha\in (1,2)$ is a fractional order of derivation and $\frac{\partial^{\alpha}}{\partial_-x^\alpha }, \frac{\partial^{\alpha}}{\partial_+x^\alpha }$ denote the left looking and right looking $\alpha$th derivatives. Discretizing in space by means of the finite difference scheme based on Gr\"unwald-Letnikov formulas, with step size $\Delta x=\frac{1}{n+1}$, yields
\[
\begin{cases}\dot{\mathbf u}(t)= A\mathbf u(t)\\
\mathbf u(0)=\mathbf{u_0}
\end{cases},\qquad A=T_n+T_n^T, \qquad T_n=\frac{1}{\Delta x^{\alpha}}\begin{bmatrix}
g_1^{(\alpha)} &g_0^{(\alpha)} &0&\dots&0&0\\
g_2^{(\alpha)} &g_1^{(\alpha)} &g_0^{(\alpha)}&0&\dots&0\\
\vdots&\ddots&\ddots&\ddots&\ddots&\vdots\\
\vdots&\ddots&\ddots&\ddots&\ddots&0\\
g_{n-1}^{(\alpha)}&\ddots&\ddots&\ddots&g_1^{(\alpha)}&g_0^{(\alpha)}\\
g_{n}^{(\alpha)}&g_{n-1}^{(\alpha)}&\dots&\dots&g_2^{(\alpha)}&g_1^{(\alpha)}
\end{bmatrix},
\]
where 
\[
g_0^{(\alpha)}=-1,\qquad g_k^{(\alpha)}=\frac{(-1)^{k+1}}{k!}\alpha(\alpha-1)\cdots(\alpha-k+1), \quad k=1,\dots,n,
\]
 and $\mathbf u(t),\mathbf{u_0}\in\mathbb R^n$ contain the sampling of the solution and of the boundary condition, respectively, at the spatial points $j\Delta x$, for $j=1,\dots,n$. In particular, evaluating the solution at time $t=1$ as $\mathbf u(1)=e^{A}\mathbf{u_0}$ requires the computation of the matrix exponential of $A$ which is well approximated in the HSS format \cite{massei2019fast}. 
 
 Concerning the latter task, we compare the performances of our D\&C method (Algorithm~\ref{alg:D&C}) with polynomial Krylov subspaces and of the function \texttt{expm} of the \texttt{hm-toolbox} that makes use of a Pad\'e approximant combined with scaling and squaring. 
 
 The results are reported in Table~\ref{tab:fractional}. 
 The column labeled as ``Dense" corresponds to the evaluation of the matrix exponential with dense arithmetic via Matlab's \texttt{expm} function. This has been computed up to size $n=8192$ and demonstrates that D\&C is slightly more accurate; \texttt{expm} (HSS) and D\&C are cheaper than the dense method from sizes $4096$ and $2048$, respectively.
  \begin{table}[htb]
 	\centering
 	\begin{filecontents*}{data/testfractional.dat}
512	0.5742	0.085873	3.7785e-08	1.8914e-08	10	13	0.033961
1024	1.0072	0.17972	6.7469e-08	2.3432e-08	11	15	0.14823
2048	1.7731	0.38194	9.8824e-08	3.998e-08	13	23	0.95617
4096	3.9911	1.0898	1.44e-07	4.8503e-08	14	23	8.9443
8192	10.221	3.1149	1.1603e-07	6.0869e-08	15	25	70.748
16384	18.475	8.6123	0	0	15	26	0
32768	37.222	22.178	0	0	16	27	0
\end{filecontents*}
\pgfplotstabletypeset[
 	every head row/.style={
 		before row={
 			\toprule
 			\multicolumn{2}{c|}{$A$} &\multicolumn{2}{c|}{D\&C}
 			&\multicolumn{2}{c|}{\texttt{expm} (HSS)}& Dense&$e^A$ \\
 		},
 		after row = \midrule,
 	},
 	every last row/.style={ after row=\bottomrule},
 	highlightcell/.style={
        postproc cell content/.append code={
          \ifnum\pgfplotstablerow=#1\relax%
              \pgfkeysalso{@cell content/.add={$\bf}{$}}
          \fi
        }
    },
 	columns = {0,5,2,4,1,3,7,6},
 	columns/0/.style = {column name = Size},
 	columns/5/.style = {column name = HSS rank, column type=c|},
 	columns/1/.style = {column name = Time, fixed  },
 	columns/3/.style = {column name = Err, column type=c|, skip rows between index={5}{7}},
 	columns/2/.style = {column name = Time, fixed, highlightcell={2}, highlightcell={3}, highlightcell={4}, highlightcell={5}, highlightcell={6} },
 	columns/4/.style = {column name = Err, skip rows between index={5}{7}, column type =c|},
 	columns/7/.style = {column name = Time, skip rows between index={5}{7}, fixed, highlightcell={0}, highlightcell={1}},
 	columns/6/.style = {column name = HSS rank, column type=|c},
 	]{data/testfractional.dat}
 	\caption{Computation of $e^A$ in the HSS format for the coefficient matrix $A$ of the fractional diffusion problem discussed in Section~\ref{sec:fractional}. We compare the performances of the \texttt{expm} function of the \texttt{hm-toolbox}~\cite{massei2020hm} and of the D\&C approach proposed in Algorithm~\ref{alg:D&C}.}
 	\label{tab:fractional}
 \end{table}

\subsection{Sampling from a Gaussian Markov random field}\label{sec:sampling}
This case study, taken from \cite{ilic2010restarted}, arises from computational statistics and it concerns a tool often used to model spatially structured uncertainty in the data. Given a cloud of points $\{s_i\}_{i=1}^n\subset\mathbb R^d$ we introduce Gaussian random variables $x_i$ $i=1,\dots,n$ at each point.  The vector $ x=(x_i)$  is referred to as a \emph{Gaussian Markov random field (GMRF)} when it is distributed according to the precision (inverse covariance) matrix $A=(a_{ij})\in\mathbb R^{n\times n}$ depending on two positive parameters $\phi$ and $\delta$ as follows:
\[
a_{ij} = \begin{cases}
1+\phi\cdot \sum\limits_{k=1,k\neq i}^n\chi_{ki}^\delta &\text{if $i= j$}\\
-\phi \cdot\chi_{ij}^\delta&\text{if $i\neq j$}
\end{cases},\qquad 
\text{where } \chi_{ij}^\delta =\begin{cases}
1&\text{if } ||s_i-s_j||_2< \delta\\
0&\text{otherwise}
\end{cases}.
\]
A sample $ v\in\mathbb R^n$ from  a  zero-mean  GMRF  with  precision  matrix $A$ is obtained as   $ v=A^{-\frac 12} z$, where $ z$ is a vector of independently and identically distributed standard normal random variables. 

When many samples are needed, it is convenient to store an HSS representation of $A^{-\frac 12}$ so that each sample requires only a matrix vector product with an HSS matrix. In this experiment  we set $\phi =3$, we generate $n=2^j$  pseudorandom points $s_i$ in the unit interval $(0,1)$, and we choose $\delta =0.02\cdot 2^{9-j}$ for $j=9,\dots,18$. Sorting the points $s_i$ yields  precision matrices  that are symmetric, diagonally dominant and with bandwidth in the range $[19,26]$.

As suggested in Remark~\ref{rmk:spdrankb}, as the matrix $A$ is banded and SPD we use a decomposition which features rank-$b$ updates; we observed a speed up with respect to doing rank-$2b$ updates in our experiments. 
In Algorithm~\ref{alg:kryl} the projection method used for computing the updates in the D\&C is the \emph{Extended Krylov method}, which alternates poles $0$ and $\infty$\footnote{More precisely, the $m$th extended Krylov subspace associated to a matrix $A$ and a (block) vector $B$ is 
$A^{-m}\mathcal{K}_{2m}(A, B) := \mathrm{span}\big[B, A^{-1}B, AB, \ldots, A^{m-1}B,A^{-m}B\big]$.}. We compare the computation of $A^{-\frac 12}$ in the HSS  format by means of our D\&C scheme with the function \texttt{sqrtm} contained in the \texttt{hm-toolbox}~\cite{massei2020hm} which combines the Denman and Beavers iteration with the HSS arithmetic.
 
 The results reported in Table~\ref{tab:sampling} show that the D\&C approach yields a significant reduction of the computational time with respect to \texttt{sqrtm} (HSS). For the largest instance, $n= 32, 768$,  we have profiled the computing time spent  at  the  different  stages  of  the  D\&C  method. The generation of the bases of the extended Krylov subspaces consumed about $25$\% of the total time while about $50$\% was spent to sum the (low-rank) updates to the block diagonal intermediate results. Around $20$\% was used for computing the projected matrices and evaluating the inverse square roots of the diagonal blocks at the lowest level of recursion and of the projected matrices. 
 
 \begin{table}[htb]
 	\centering
 			\begin{filecontents*}{data/testsampling.dat}
512	0.48616	0.051202	3.436e-09	2.0213e-09	22	7	14	0.017102
1024	1.408	0.15613	5.2224e-09	3.4491e-09	20	8	17	0.12931
2048	3.9931	0.36625	6.3784e-09	3.7646e-09	19	10	19	0.94742
4096	9.0545	0.80165	5.6108e-09	3.2324e-09	21	7	19	9.0289
8192	21.266	2.461	6.6112e-09	3.4628e-09	22	8	20	70.422
16384	48.915	5.6962	0	0	25	9	22	0
32768	102.65	15.119	0	0	26	11	25	0
65536	209.56	26.251	0	0	26	10	24	0
1.3107e+05	417.21	60.442	0	0	25	10	24	0
2.6214e+05	918.81	146.97	0	0	26	11	26	0
\end{filecontents*}
\pgfplotstabletypeset[
 		every head row/.style={
 			before row={
 				\toprule
 				\multicolumn{2}{c|}{$A$} &\multicolumn{2}{c|}{D\&C}
 		&\multicolumn{2}{c|}{\texttt{sqrtm} (HSS)}& Dense &$A^{-\frac12}$ \\
 			},
 			after row = \midrule,
 		},
 		every last row/.style={ after row=\bottomrule},
        highlightcell/.style={
            postproc cell content/.append code={
              \ifnum\pgfplotstablerow=#1\relax%
                  \pgfkeysalso{@cell content/.add={$\bf}{$}}
              \fi
            }
        },
 		columns = {0,5,2,4,1,3,8,7},
 		columns/0/.style = {column name = Size, fixed},
 		columns/5/.style = {column name = Band, column type=c|},
 		columns/1/.style = {column name = Time, fixed },
 		columns/3/.style = {column name = Err, column type=c|, skip rows between index={5}{10}},
 		columns/2/.style = {column name = Time, fixed,  highlightcell={2}, highlightcell={3}, highlightcell={4}, highlightcell={5}, highlightcell={6}, highlightcell={7}, highlightcell={8}, highlightcell={9} },
 		columns/4/.style = {column name = Err, skip rows between index={5}{10}, column type = c|},
 		columns/8/.style = {column name = Time, skip rows between index={5}{10}, fixed, zerofill, highlightcell={0}, highlightcell={1}},
 		columns/7/.style = {column name = HSS rank, column type=|c},
 		]{data/testsampling.dat}
 	\caption{Computation of $A^{-\frac 12}$ in the HSS format for the precision matrix $A$ of the Gaussian Markov random field discussed in Section~\ref{sec:sampling}. We compare the performances of the \texttt{sqrtm} function of the \texttt{hm-toolbox}~\cite{massei2020hm} and of the D\&C approach proposed in Algorithm~\ref{alg:D&C}.}
 	\label{tab:sampling}
 \end{table}

 \subsection{Merton model for option pricing}\label{sec:toeplitz}
 We consider the evaluation of option prices in the Merton model for one single underlying asset, as in~\cite[Section 6.3]{KressnerLuce2018}.  More specifically, we compute the exponential of the nonsymmetric Toeplitz matrix $A$ arising from the discretization of the partial integro-differential equation
 \begin{equation*}
  \omega_t = \frac{\nu^2}{2}\omega_{\xi\xi} + \left ( r - \lambda\kappa - \frac{\nu^2}{2} \right ) \omega_{\xi} - (r + \lambda)\omega + \lambda\int_{-\infty}^{+\infty}\omega(\xi + \eta, t)\phi(\eta)\mathrm{d}\eta,
 \end{equation*}
 where $\omega(\xi, t)$ on $(-\infty,+\infty) \times [0, T]$ is the option value, $T$ is the time to maturity, $\nu \ge 0$ is the volatility, $r$ is the risk-free interest rate, $\lambda \ge 0$ is the arrival intensity of a Poisson process, $\phi$ is the normal distribution with mean $\mu$ and standard deviation $\sigma$, and $\kappa = e^{\mu + \sigma^2/2}-1$. We use the same discretization and parameters as~\cite[Section 6.3]{KressnerLuce2018} and~\cite[Example 3]{Lee2010}.
 

We aim at approximating $\exp(A)$, for different values of the matrix size $n$. To do so, we first convert $A$ into HSS format using the hm-toolbox~\cite{massei2020hm}, then rescale it by dividing by $ 2^{\left \lceil \log_2 \|H\|_2 \right \rceil}$, then apply Algorithm~\ref{alg:D&C}, and finally squaring the result $ \left \lceil \log_2 \|H\|_2 \right \rceil$ times in the HSS format. We use 
polynomial Krylov subspaces for the updates in Algorithm~\ref{alg:D&C}. For different values of $n$, we compare the output of the described method with the \texttt{expm} algorithm from the hm-toolbox~\cite{massei2020hm} and the algorithm~\texttt{sexpmt} proposed in~\cite{KressnerLuce2018}. In order to attain a similar accuracy to the \texttt{sexpmt} algorithm, we set the tolerance parameter $\varepsilon = 10^{-12}$ in Algorithm~\ref{alg:D&C} and for HSS computations in the hm-toolbox~\cite{massei2020hm}. The results are summarized in Table~\ref{tab:toeplitz}. 
 
 \begin{table}[ht]
 	\centering
 	\begin{filecontents*}{data/testToeplitz.dat}
512	0.56914	2.3353e-10	0.49053	2.6612e-11	0.1583	2.9488e-12	18	0.10033
1024	1.8179	5.9899e-10	0.86392	7.1316e-10	0.53627	2.0437e-11	18	0.68516
2048	4.1562	8.0296e-09	2.3996	1.1773e-09	1.3064	3.3713e-11	17	5.0491
4096	8.0604	2.9589e-08	5.5261	6.1896e-08	7.3884	1.8628e-10	19	42.33
8192	16.227	3.1048e-07	9.602	5.6482e-07	25.523	1.3485e-09	18	323.18
16384	33.711	0	20.575	0	98.413	0	20	0
32768	67.427	0	46.127	0	419.82	0	20	0
\end{filecontents*}
\pgfplotstabletypeset[
 	every head row/.style={
 		before row={
 			\toprule
 			{$A$} &\multicolumn{2}{c|}{D\&C}
 			&\multicolumn{2}{c|}{\texttt{expm} (HSS)}& \multicolumn{2}{c|}{\texttt{sexpmt}}&Dense&$e^A$ \\
 		},
 		after row = \midrule,
 	},
 	every last row/.style={ after row=\bottomrule},
 	highlightcell/.style={
        postproc cell content/.append code={
          \ifnum\pgfplotstablerow=#1\relax%
              \pgfkeysalso{@cell content/.add={$\bf}{$}}
          \fi
        }
    },
 	columns = {0,3,4,1,2,5,6,8,7},
 	columns/0/.style = {column name = Size, column type = c|},
 	columns/1/.style = {column name = Time, fixed },
 	columns/2/.style = {column name = Err, column type=c|, skip rows between index={5}{7}},
 	columns/3/.style = {column name = Time, fixed, highlightcell={3}, highlightcell={4}, highlightcell={5}, highlightcell={6} },
 	columns/4/.style = {column name = Err, skip rows between index={5}{7}, column type =c|},
 	columns/5/.style = {column name = Time, fixed, highlightcell={1}, highlightcell={2} },
 	columns/6/.style = {column name = Err, skip rows between index={5}{7}, column type =c|},
 	columns/8/.style = {column name = Time, skip rows between index={5}{7}, fixed, highlightcell={0}},
 	columns/7/.style = {column name = HSS rank, column type=|c},
 	]{data/testToeplitz.dat}
 	\caption{Computation of $e^A$ in the HSS format for the coefficient matrix $A$ in Section~\ref{sec:toeplitz}. We compare the performances of our Algorithm~\ref{alg:D&C} with the \texttt{expm} function of the \texttt{hm-toolbox}~\cite{massei2020hm} and the \texttt{sexpmt} algorithm of~\cite{KressnerLuce2018}.}
 	\label{tab:toeplitz}
 \end{table}
 
  \subsection{Neumann-to-Dirichlet operator}\label{sec:NtD}
  Consider 
  \begin{equation}\label{eq:ntd}
   \frac{\partial^2}{\partial x^2}u = Au, \qquad \frac{\partial}{\partial x} u \mid_{x=0} = -b, \qquad u\mid_{x = + \infty}
  \end{equation}
  for a nonsingular matrix $A$ which is the discretization of a differential operator on some spatial domain $\Omega \subseteq \R^{\ell}$. Then~\eqref{eq:ntd} is a semidiscretization of an $(\ell+1)$-dimensional PDE on $[0, +\infty) \times \Omega$; the solution is given by $u(x) = \exp \left ( - xA^{-1/2}\right ) A^{-1/2}b$. In particular, $u(0) = A^{-1/2}b$ and the operator $A^{-1/2}$ is called \emph{Neumann-to-Dirichlet} (NtD) operator as it allows for conversion of the Neumann data $-b$ at the boundary $x = 0$ into the Dirichlet data $u(0)$, without needing to solve~\eqref{eq:ntd} on its unbounded domain.

  As in~\cite[Example 6.1]{Druskin2016}, we consider the inhomogeneous Helmholtz equation
  \begin{equation}\label{eq:exampleNtD}
   \Delta u(x, y) + k^2 u(x, y) = f(x, y), \quad f(x, y) = 10\delta(x - 511\pi / 512) \delta(y-50\pi/512)
  \end{equation}
  for $k = 50$ on the domain $[0, \pi]^2$. The matrix $A$ corresponds to the discretization of $-\frac{\partial^2}{\partial y^2} - k^2$ on $[0, \pi]$ by central finite differences. We consider step sizes $ h \in \{\pi/2^9, \ldots, \pi/2^{15}\}$ and compute the NtD operator $A^{-1/2}$ in the HSS format using the D\&C algorithm~\ref{alg:D&C}; 
  Table~\ref{tab:NtD} illustrates the comparison with the computation of $A^{-1/2}$ in dense arithmetic. For computing the inverse square root, we move the branch cut to the negative imaginary axis. For the updates, we use the complex extension of Algorithm~\ref{alg:kryl} with rational Krylov subspaces where we cyclically repeat $6$ poles coming from the degree-6 approximation to $f(z) = z^{-1/2}$ on the set $ S := [-b, -a] \cup [a, b]$ for $b = \|A\|_2$ (estimated with \texttt{normest(A)}) and $a = 1/\|A^{-1}\|_2$ (computed via \texttt{b / condest(A)}) described in~\cite[Section 2]{Druskin2016}. As the spectral interval of $A$ contains zero, which is a singularity of the inverse square root function, we could potentially encounter instability issues when using Krylov subspace methods; however, this does not happen in our example. 
  
 
  \begin{table}[ht]
 	\centering
 	\begin{filecontents*}{data/testNtD.dat}
512	1.1812	2.5263e-08	0.13226	12
1024	0.99298	3.5911e-08	0.20922	13
2048	1.8438	4.2583e-08	0.92005	20
4096	3.2252	6.387e-08	6.7319	20
8192	7.8897	8.1952e-08	64.744	20
16384	16.749	0	0	20
32768	37.701	0	0	20
\end{filecontents*}
\pgfplotstabletypeset[
 	every head row/.style={
 		before row={
 			\toprule
 			{$A$} 
 			&\multicolumn{2}{c|}{D\&C}& Dense&$A^{-1/2}$ \\
 		},
 		after row = \midrule,
 	},
 	every last row/.style={ after row=\bottomrule},
 	highlightcell/.style={
        postproc cell content/.append code={
          \ifnum\pgfplotstablerow=#1\relax%
              \pgfkeysalso{@cell content/.add={$\bf}{$}}
          \fi
        }
    },
 	columns = {0,1,2,3,4},
 	columns/0/.style = {column name = Size, column type = c|},
 	columns/1/.style = {column name = Time, fixed , highlightcell={3}, highlightcell={4}, highlightcell={5}, highlightcell={6}},
 	columns/2/.style = {column name = Err, column type=c|, skip rows between index={5}{7}},
 	columns/3/.style = {column name = Time, skip rows between index={5}{7}, fixed, highlightcell={0}, highlightcell={1}, highlightcell={2}},
 	columns/4/.style = {column name = HSS rank, column type=|c},
 	]{data/testNtD.dat}
 	\caption{Computation of $A^{-1/2}$ in the HSS format for Neumann-to-Dirichlet problem discussed in Section~\ref{sec:NtD}.}
 	\label{tab:NtD}
 \end{table}
 
 
 \subsection{Computing charge densities}\label{sec:densities}
 The approximation of the diagonal of a matrix function applies to the calculation of the electronic structure of systems of atoms.  In particular, the charge densities of a system are contained in the diagonal of $f(H)$, where $f$ is the Heaviside function
 $$
 f(x)= \begin{cases}
 1&x<0\\
 0&x\geq 0
 \end{cases}
 $$
  and $H$ is the Hamiltonian matrix that is given by   the sum of the kinetic and potential energies. The entries of Hamiltonian matrices usually  decay rapidly away from the main diagonal. Let us consider the parametrized model Hamiltonian given in \cite[Section 4.3]{bekas2007estimator}:
 \[
 H\in\mathbb R^{N_b\cdot N_s\times N_b\cdot N_s},\qquad H_{N_b\cdot (i-1)+j, i'\cdot N_b(i'-1)+j'}=\begin{cases}(i-1)\Delta+(j-1)\delta&i=i',\ j=j'\\
 C\cdot e^{-|j-j'|}&i=i',\ j\neq j'\\
 \frac{C}{n_{od}(|i-i'|+1)}\cdot e^{-|j-j'|}&\text{otherwise}
 \end{cases},
 \]
 where we have set the parameters' values: $N_b=5,N_s=1600,\Delta=10^{-1},\delta=10^{-4},C=10^{-1}$, and $n_{od}=5000$. The HSS structure of the matrix $H$ is shown in the left part of Figure~\ref{fig:hamilton}. We compute the diagonal of $f(H)$ by means of Algorithm~\ref{alg:D&C} and exploiting the relation
 $$
 f(x)= (1-\text{sign}(x))/2.
 $$
 More specifically, we use Algorithm~\ref{alg:D&C} to compute the diagonal of $\text{sign}(H)$; then we subtract the latter from the vector of all ones and we divide by $2$. 
 The procedure has terminated after $3.52$ seconds. As benchmark method we evaluate $f(H)$ by diagonalization with dense arithmetic. This has required $78.62$
 seconds. The Euclidean distance of the vectors obtained with the two approaches is $2.68\cdot 10^{-11}$. In Figure~\ref{fig:hamilton}, the first $500$ components of the two charge densities are shown.
 \begin{figure}[ht!]
 	\centering
 	\includegraphics[scale=.4]{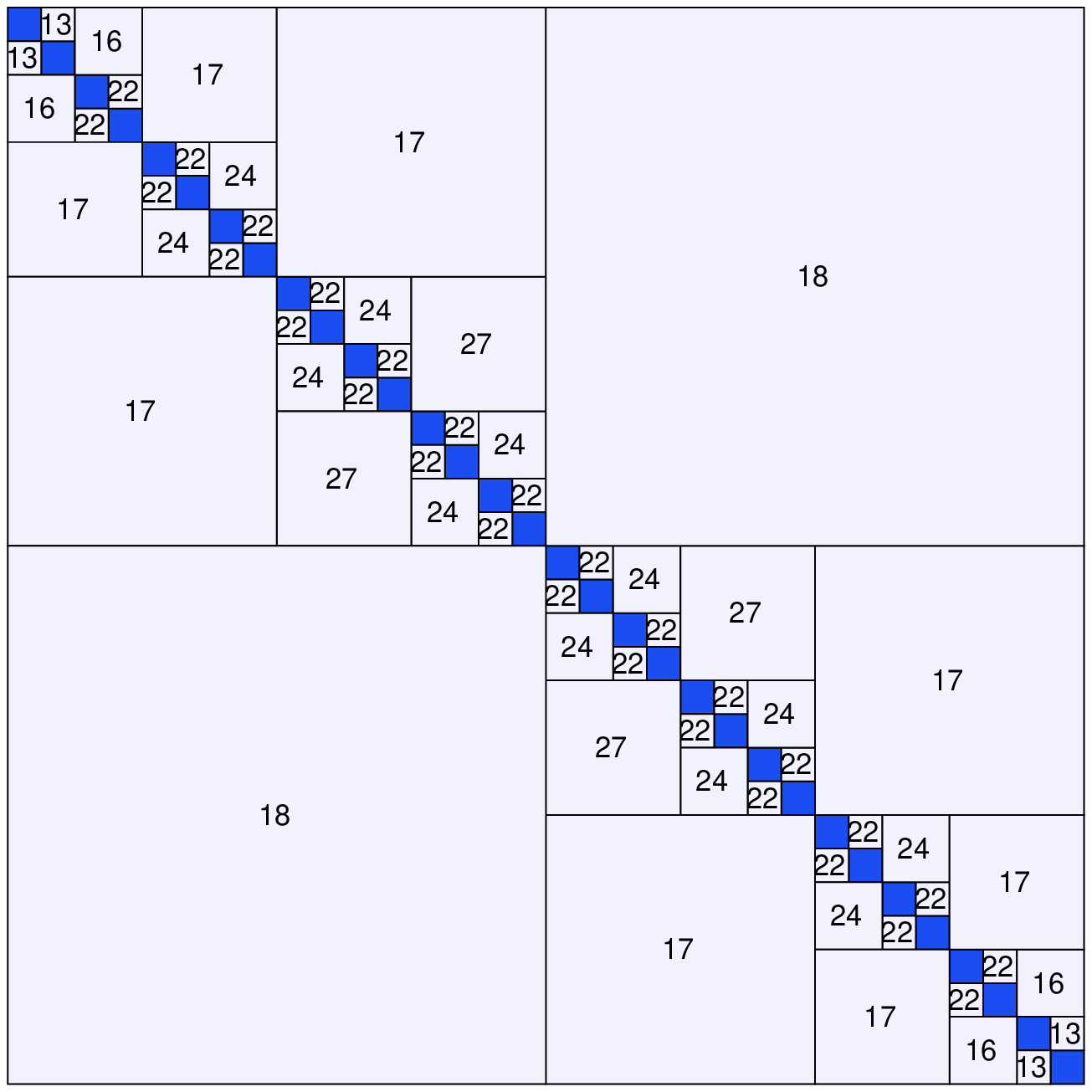}~\includegraphics[scale=.55]{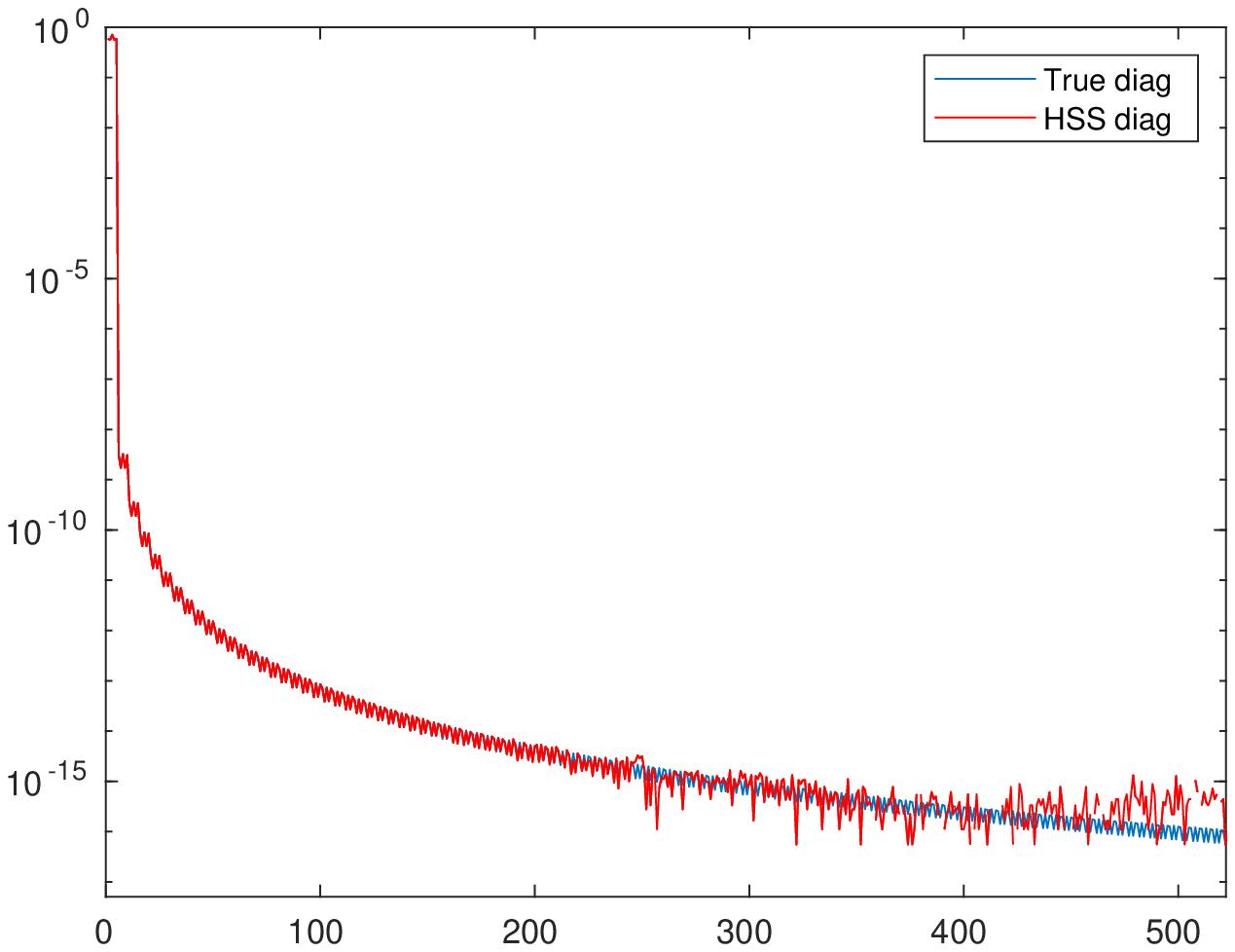}
 	\caption{Left: Ranks of the off-diagonal blocks of the Hamiltonian matrix $H$ from Section~\ref{sec:densities}; the blue blocks indicate matrices for which dense arithmetics is used. Right: Charge densities estimated with dense arithmetic (blue) and with the HSS D\&C method (red).}
 	\label{fig:hamilton}
 \end{figure}
 
 \subsection{Computing subgraph centralities and Estrada index}\label{sec:estrada}
 Given an undirected graph $\mathcal{G}$ with adjacency matrix $A$, the diagonal entries of $\exp(A)$ are called the \emph{subgraph centralities} of the vertices. Their normalized sum $EE_n(\mathcal{G}) := \frac{1}{n} \trace(\exp(A))$ is called the normalized Estrada index of the graph; it was introduced in~\cite{Estrada2000} to characterize the folding of molecular structures and has found applications in network analysis~\cite{Estrada2010}. 
 
 When aiming at the diagonal of $\exp(A)$, at each step of our D\&C method we run a clustering algorithm~\cite{metis} on the matrix to divide it into two components that have few edges between them; the \texttt{ufactor} parameter is set to 100. If the rank of the off-diagonal part is less than $1/15$ of the matrix size, we compute a low-rank update, otherwise we use the \texttt{mmq} algorithm~\cite{GolubMeurant2010}, which approximates each diagonal entry of $\exp(A)$ by Gauss quadrature.
 We also use \texttt{mmq} on matrices of size less than $n_{\min} = 256$. We compare our D\&C algorithm for the diagonal with \texttt{mmq} and \texttt{diag(expm(full(A)))}.
 
 When aiming at $\mathrm{trace}(\exp(A))$, we use \texttt{sum(exp(eig(full(A))))} instead of \texttt{mmq} to address small blocks or blocks that cannot be divided in smaller blocks with a low-rank correction; we noticed that this is faster than letting Matlab work with the matrices in sparse format. As a competitor for the computation of the trace we consider \texttt{sum(exp(eig(full(A))))}.
 
  In Table~\ref{tab:newEstradaDiag} we report the errors and the time needed by our algorithm. The matrices we used are \texttt{minnesota}, \texttt{power}, \texttt{as-735}, \texttt{nopoly}, \texttt{worms20\_10NN},  and \texttt{fe\_body} from the SuiteSparse Matrix Collection~\cite{UFL}.

 \begin{table}[htb]
 	\centering
 	\begin{filecontents*}{data/testDiagGraphsNew.dat}
2642	1.0142	6.2359e-10	0.80257	1.8164e-11	1.9788	0.14416	7.7095e-13	0.44313
4941	2.0601	1.2941e-08	5.152	3.3902e-11	16.105	0.46545	7.7543e-11	3.6103
7716	8.0057	4.0312e-09	24.192	2.2928e-10	56.587	3.9144	1.9631e-12	8.7258
10774	15.874	1.0381e-08	39.417	3.5379e-10	151.52	2.9792	2.687e-10	21.038
20055	38.488	2.5854e-09	97.526	1.4004e-11	929.25	6.985	2.6562e-13	124.34
45087	182.19	-1	603.57	-1	-1	27.987	-1	-1
\end{filecontents*}
\pgfplotstabletypeset[
 	every head row/.style={
 		before row={
 			\toprule
 			{$A$} 
 			&\multicolumn{2}{c|}{D\&C diagonal}& \multicolumn{2}{c|}{mmq diagonal}&{\texttt{expm}} & \multicolumn{2}{c|}{D\&C trace}& {\texttt{eig}}\\
 		},
 		after row = \midrule,
 	},
 	every last row/.style={ after row=\bottomrule},
 	highlightcell/.style={
        postproc cell content/.append code={
          \ifnum\pgfplotstablerow=#1\relax%
              \pgfkeysalso{@cell content/.add={$\bf}{$}}
          \fi
        }
    },
 	columns = {0,1,2,3,4,5,6,7,8},
 	columns/0/.style = {column name = Size, column type = c|},
 	columns/1/.style = {column name = Time, fixed, highlightcell={2}, highlightcell={3}, highlightcell={4}, highlightcell={5}, highlightcell={1} },
 	columns/2/.style = {column name = Err, skip rows between index={5}{6}, column type =c|},
 	columns/3/.style = {column name = Time, fixed, highlightcell={0} },
 	columns/4/.style = {column name = Err, skip rows between index={5}{6}, column type =c|},
 	columns/5/.style = {column name = Time, skip rows between index={5}{6}, column type = c|},
 	columns/6/.style = {column name = Time, fixed, highlightcell={0}, highlightcell={1}, highlightcell={2}, highlightcell={3}, highlightcell={4}, highlightcell={5}},
 	columns/7/.style = {column name = Err, skip rows between index={5}{6}, column type =c|},
 	columns/8/.style = {column name = Time,  skip rows between index={5}{6},column type=c},
 	]{data/testDiagGraphsNew.dat}
 	\caption{Computation of the diagonal and the trace of $e^A$ for the graphs from Section~\ref{sec:estrada}.}
 	\label{tab:newEstradaDiag}
 \end{table}

 \subsubsection{The lag parameter}\label{sec:lag}
 We compare the timings and the accuracy of our D\&C algorithm on the matrices \texttt{nopoly} and \texttt{worms20\_10NN} for values of the lag parameter in the range $\{1, 2, 3, 4\}$. 
 The results are reported in Table~\ref{tab:lag}. In general, it looks like we can safely put the lag parameter equal to $1$.  
 
   \begin{table}[ht]
 	\centering
 	\begin{filecontents*}{data/testLag.dat}
1	19.441	2.8142	1.0381e-08	2.6873e-10	47.817	6.3707	2.5854e-09	2.6007e-13
2	16.542	3.1266	1.0029e-08	3.775e-12	61.755	12.04	2.4647e-09	5.1946e-16
3	14.634	4.0013	9.2911e-09	2.4085e-14	87.849	11.55	2.3383e-09	5.1946e-16
4	16.076	3.7595	8.5132e-09	1.4196e-14	89.007	20.858	2.1559e-09	1.7315e-16
\end{filecontents*}
\pgfplotstabletypeset[
 	every head row/.style={
 		before row={
 			\toprule
 			{} 
 			&\multicolumn{4}{c|}{\texttt{nopoly}}&\multicolumn{4}{c}{\texttt{worms20\_10NN}}\\
 		},
 		after row = \midrule,
 	},
 	every last row/.style={ after row=\bottomrule},
 	highlightcell/.style={
        postproc cell content/.append code={
          \ifnum\pgfplotstablerow=#1\relax%
              \pgfkeysalso{@cell content/.add={$\bf}{$}}
          \fi
        }
    },
 	columns = {0,1,2,3,4,5,6,7,8},
 	columns/0/.style = {column name = Lag, column type = c|},
 	columns/1/.style = {column name = Diag, fixed, highlightcell={2} },
 	columns/2/.style = {column name = Trace, column type =c|, highlightcell={0}},
 	columns/3/.style = {column name = Err diag },
 	columns/4/.style = {column name = Err trace, column type = c| },
    columns/5/.style = {column name = Diag, fixed, highlightcell={0} },
 	columns/6/.style = {column name = Trace, column type =c|, highlightcell={0}},
 	columns/7/.style = {column name = Err diag },
    columns/8/.style = {column name = Err trace },
 	]{data/testLag.dat}
 	\caption{For two matrices from~\cite{UFL} we investigate the influence of the lag parameter on the timing of the D\&C algorithm for computing the diagonal and the trace of $\exp(A)$. }
 	\label{tab:lag}
 \end{table}
 
 
    \section{Block diagonal splitting algorithm for banded matrices}\label{sec:splitting}

As already mentioned in Remark~\ref{rmk:bandedpolynomial} and shown in more detail below, Algorithm~\ref{alg:D&C} applied to a banded matrix returns again a banded matrix when polynomial Krylov subspace bases are used. The purpose of this section is to go further and use this observation to bypass the need for building Krylov subspaces. We can also avoid recursion and arrive at a simpler algorithm.  

\subsection{Block diagonal splitting algorithm from low-rank updates}\label{sec:derivation}

Let $A \in \R^{n \times n}$ be banded with bandwidth $b$.
As in Section~\ref{sec:general} we start with a partitioning
 \begin{equation} \label{eq:yetanotherpart}
 A = D + R, \quad D = \begin{bmatrix}
 D_1 \\
 & \widetilde{D_1}
 \end{bmatrix}, \quad D_1 \in \R^{s\times s}, \quad R = A-D,
 \end{equation}
 but we now suppose that the first diagonal block is small, that is, $s\ll n$; see also Figure~\ref{fig:DE}.
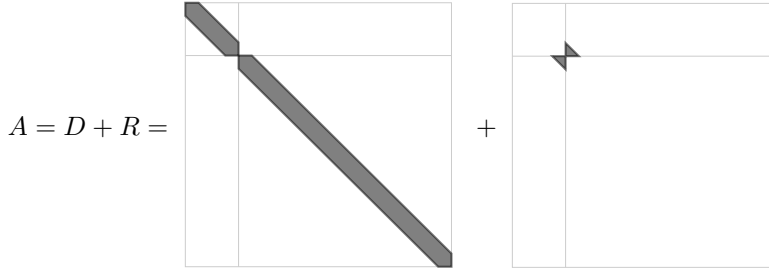
\begin{figure}[ht]
\begin{equation*} 
A = D + R =  \raisebox{-50pt}{
    \begin{tikzpicture}[scale=0.035]
        \draw[very thin, gray!40!white, step=20] (0, 0) rectangle (100, 100);
        \draw[very thin, gray!40!white, step=20] (20, 0) -- (20, 100);
        \draw[very thin, gray!40!white, step=20] (0, 80) -- (100, 80);        
        \draw[fill = black, opacity = 0.5, thick] (0, 100-0) -- (0+5, 100-0) -- (0+20, 85-0) -- (0+20, 80-0) -- (0+15, 80-0) -- (0, 95-0) -- cycle;
        \draw[fill = black, opacity = 0.5, thick] (20, 100-20) -- (20+5, 100-20) -- (80+20, 85-80) -- (80+20, 80-80) -- (80+15, 80-80) -- (20, 95-20) -- cycle;
    \end{tikzpicture}}
     \,\,\,\,+ \raisebox{-50pt}{
    \begin{tikzpicture}[scale=0.035]
        \draw[very thin, gray!40!white, step=20] (0, 0) rectangle (100, 100);
        \draw[very thin, gray!40!white, step=20] (20, 0) -- (20, 100);
        \draw[very thin, gray!40!white, step=20] (0, 80) -- (100, 80); 
        \draw[fill = black, opacity = 0.5, thick] (20, 95-20) -- (20, 105-20) -- (20+5, 100-20) -- (20-5, 100-20) -- cycle;
    \end{tikzpicture}}
\end{equation*}
\caption{Illustration of decomposition~\eqref{eq:yetanotherpart}.}
\label{fig:DE}
\end{figure}

The matrix $R$ can be written as $R = U_1 J U_1^T$ where 
\begin{equation*}
U_1 :=
\begin{tikzpicture}[mymatrixenv,baseline={(0,-0.1)}]
    \matrix[mymatrix] (m)  {
        0 & ~ & I_{2b} & ~ & ~ & 0  &       ~ & ~ \\
    };
    \mymatrixbracebottom{1}{1}{$s-b$}
    \mymatrixbracebottom{3}{3}{$2b$}
    \mymatrixbracebottom{4}{8}{$n-s-b$}
\end{tikzpicture}^T \text{ and }J := \begin{bmatrix}
                               & A(\texttt{s-b+1:s,s+1:s+b}) \\
                               A(\texttt{s+1:s+b, s-b+1:s}) & \\
                             \end{bmatrix}.
                             \end{equation*}
When applying Algorithm~\ref{alg:kryl} to approximate the low-rank update $f(A)-f(D)$ the polynomial Krylov subspaces remain sparse in the following sense.
\begin{lemma}\label{lemma:Um}
Given the setting described above, assume that $2mb \le s$. Then the Krylov subspaces $\mathcal{K}_m(D, U_1)$ and $\mathcal{K}_m(D^T, U_1)$ are each contained in the column span of the $n \times 2mb$ matrix
 \begin{equation*}
  U_m := \begin{tikzpicture}[mymatrixenv,baseline={(0,-0.1)}]
    \matrix[mymatrix] (m)  {
        0 & ~ & I_{2mb} & ~ & ~ & 0 & ~ & ~ \\
    };
    \mymatrixbracebottom{1}{1}{$s-mb$}
    \mymatrixbracebottom{3}{3}{$2mb$}
    \mymatrixbracebottom{4}{8}{$n-s-mb$}
\end{tikzpicture}^T.
 \end{equation*}
\end{lemma}
\begin{proof}
For every polynomial $p \in \Pi_{m-1}$, the matrix $p(D)$ is banded with bandwidth $(m-1) b$. In turn, $p(D)U_1$ only has nonzero rows at positions $s-mb+1, \ldots, s+mb$ or, in other words, every column of $p(D)U_1$ is contained in the column span of $U_m$. Combined with the definition $\mathcal{K}_m(D, U_1) = \mathrm{span}[U_1,DU_1,\ldots,D^{m-1}U_1]$, this proves the statement of the lemma.
\end{proof}

The compressions of $D$ and $A$ with respect to the orthonormal basis $U_m$ from Lemma~\ref{lemma:Um} takes the form
\begin{align*}
G_m = U_m^T D U_m & = \text{blkdiag}(A(\texttt{s-mb+1:s, s-mb+1:s}), A(\texttt{s+1:s+mb, s+1:s+mb})))\\
& =: \text{blkdiag}(C_1^{(1)}, C_1^{(2)}), \\
    H_m = U_m^T A U_m & = A(\texttt{s-mb+1:s+mb, s-mb+1:s+mb}) =: B_1.
\end{align*}
Following Algorithm~\ref{alg:kryl}, we define the approximate low-rank update as
\begin{align}
 f(A) - f(D) & = f(A) - \text{blkdiag}(f(D_1), f(\widetilde{D_1}))
 \nonumber \\ & \approx U_m f(B_1) U_m^T - U_m f(\text{blkdiag}(C_1^{(1)}, C_1^{(2)})) U_m^T. \label{eq:lala}
\end{align}
By Lemma~\ref{lemma:Um}, this approximation becomes in fact identical to the one returned by Algorithm~\ref{alg:kryl} if $\mathcal{K}_m(D, U_1)$ and $\mathcal{K}_m(D^T, U_1)$ each have dimension $2mb$. If the Krylov subspaces are of smaller dimension then the approximations may differ, but the exactness properties mentioned in~\ref{sec:traceSymmetric} still hold (see Remark~\ref{rmk:exactness}). 

\subsection{The block diagonal splitting algorithm} \label{sec:splittingalgorithm}

From~\eqref{eq:lala}, it follows that the first part of Algorithm~\ref{alg:D&C} (lines~\ref{step:fact}--\ref{line:rec1}) reduces to the computation of $f(B_1)$, $f(C_1^{(1)})$, $f(C_1^{(2)})$, $f(D_1)$, that is, functions of small submatrices of $A$. For the second part (line~\ref{line:rec2}) one can apply the same reasoning recursively to $\widetilde{D_1}$.

With the simplified assumptions that $n = ks$ for an integer $k$ and $m := \frac{s}{2b}$ is an integer, the discussion above shows that Algorithm~\ref{alg:D&C} reduces to the simpler Algorithm~\ref{alg:bandedpoly}, where
\begin{itemize}
    \item $D := \mathrm{blkdiag}(D_1, \ldots, D_k)$ and $D_1, \ldots, D_k$ are the consecutive $s\times s$ diagonal blocks of $A$;
    \item $\widetilde B := \mathrm{blkdiag}(B_1, \ldots, B_{k-1})$ and $B_1, \ldots, B_{k-1}$ are consecutive $s \times s$ diagonal blocks of $A$ starting from index $\frac{s}{2}+1$;
    \item $\widetilde C := \mathrm{blkdiag}(C_1^{(1)}, C_1^{(2)}, \ldots, C_{k-1}^{(1)}, C_{k-1}^{(2)})$ where $C_1^{(1)}, \ldots, C_{k-1}^{(2)}$ are the consecutive $\frac{s}{2} \times \frac{s}{2}$ diagonal blocks of $A$ starting from index $\frac{s}{2}+1$;
    \item $B := \mathrm{blkdiag}(Z, \widetilde B, Z)$, $C := \mathrm{blkdiag}(Z, \widetilde C, Z)$, where $Z := \texttt{zeros}(\frac{s}{2})$.
\end{itemize}
The resulting splitting $A = D + B - C$ is illustrated in Figure~\ref{fig:BCD}.

\begin{algorithm}[ht]
\caption{Approximation of $f(A)$ for banded $A$}
\label{alg:bandedpoly}
 \begin{algorithmic}[1]
  \REQUIRE{Banded matrix $A \in \R^{n \times n}$ of bandwidth $b$, block size $s$, function $f$}
  \ENSURE{Approximation $\mathrm{approx}_f^{(s)}(A)$ of $f(A)$}
  \STATE{Define $\widetilde B$, $B$, $\widetilde C$, $C$,  and split $A = D + B - C$ as explained in Section~\ref{sec:splittingalgorithm}}
  \STATE{Compute $f(D)$, $f(\widetilde B)$, and $f(\widetilde C)$ by evaluating $f$ on each block of $D$, $\widetilde B$, and $\widetilde C$}
  \STATE{Set $f_B := \mathrm{blkdiag}(Z, f(\widetilde B), Z)$ and $f_C := \mathrm{blkdiag}(Z, f(\widetilde C), Z)$, where $Z:=\texttt{zeros}(s/2)$} 
  \STATE{Return $f(D) + f_B - f_C$}
 \end{algorithmic}
\end{algorithm}

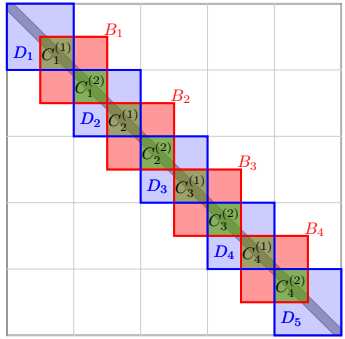
\begin{figure}[ht]
 \begin{center}
 \resizebox{0.3\textwidth}{!}{
    \begin{tikzpicture}[scale=0.07]
    \draw [thick] (0,0) rectangle (100, 100);
        \draw[very thin, gray!40!white, step=20] (0, 0) grid (100, 100);
        \draw[fill = black, opacity = 0.3] (0, 100) -- (2, 100) -- (100, 2) -- (100, 0) -- (98, 0) -- (0, 98) -- cycle;
        \foreach \x in {1, 2, 3, 4, 5}
            { \draw[fill = blue, opacity = 0.2, blue, very thick] (20*\x-20,100-20*\x) rectangle (20*\x, 120-\x*20); 
            \node[blue] at (20*\x-15, 105-20*\x) {$D_{\x}$};
            }
        \foreach \x in {1, 2, 3, 4}
            { \draw[fill = red, opacity = 0.4, red, very thick] (20*\x-10,90-20*\x) rectangle (20*\x + 10, 110-20*\x); }
        \foreach \x in {1, 2, 3, 4}
            { \draw[fill = green, opacity = 0.3, thick] (20*\x-10,100-20*\x) rectangle (20*\x, 110-20*\x);
            \draw[fill = green, opacity = 0.4, thick] (20*\x+10,100-20*\x) rectangle (20*\x, 90-20*\x);
            \node at (20*\x - 5, 105-20*\x) {$C_{\x}^{(1)}$};
            \node at (20*\x + 5, 95-20*\x) {$C_{\x}^{(2)}$};
            }
        \foreach \x in {1, 2, 3, 4}
            { \draw[red, very thick] (20*\x-10,90-20*\x) rectangle (20*\x + 10, 110-20*\x); 
            \node[red] at (20*\x+12, 112-20*\x) {$B_{\x}$};}
        \foreach \x in {1, 2, 3, 4, 5}
            { \draw[blue, very thick] (20*\x-20,100-20*\x) rectangle (20*\x, 120-\x*20); 
            \node[blue] at (20*\x-15, 105-20*\x) {$D_{\x}$};
            }
    \end{tikzpicture}}
    \end{center}
    \caption{The blocks that are involved in the computation of $f(A)$ for a banded matrix $A$. }
        \label{fig:BCD}
\end{figure}

%

\subsection{Convergence analysis of block diagonal splitting method}\label{sec:convbanded}

Algorithm~\ref{alg:bandedpoly} corresponds to Algorithm~\ref{alg:D&C} where the updates are performed using projection onto spaces that \emph{include} polynomial Krylov subspaces of dimension $m := \left \lfloor \frac{s}{2b} \right \rfloor$; thanks to Remark~\ref{rmk:exactness} and Proposition~\ref{prop:exactDC} this implies that Algorithm~\ref{alg:bandedpoly} is exact for all $f \in \Pi_m$. 
This property allows us to prove convergence results for Algorithm~\ref{alg:bandedpoly}. 
In the following, we let $W(A) := \{z^T A z \mid |z| = 1\}$ denote the numerical range of $A$. 
\begin{theorem}\label{thm:convpoly}
 Let $A \in \R^{n \times n}$ be a banded matrix with bandwidth $b$. For a given block size $s$, the output $\mathrm{approx}_f^{(s)}(A)$ of Algorithm~\ref{alg:bandedpoly} satisfies
 \begin{equation*}
  \| f(A) - \mathrm{approx}_f^{(s)}(A) \|_2 \le 4 C \min_{p \in \Pi_{m}} \| f - p \|_{W(A)},
 \end{equation*}
 where $C=1$ if $A$ is normal and $C=1+\sqrt{2}$ otherwise, and $m := \left\lfloor \frac{s}{2b} \right\rfloor$.
\end{theorem}

\begin{proof}
  Algorithm~\ref{alg:bandedpoly} is exact for a polynomial in $\Pi_{m}$ and is linear with respect to $f$, therefore for all $p \in \Pi_m$ we have
 \begin{align*}
  \| f(A) - \mathrm{approx}_f^{(s)}(A) \|_2 & = \| f(A) - p(A) + \mathrm{approx}_p^{(s)}(A) - \mathrm{approx}_f^{(s)}(A) \|_2 \\
  & = \| f(A) - p(A) - \mathrm{approx}_{f-p}^{(s)}(A) \|_2 \\
  & \le \| (f-p)(A) \|_2 + \| \mathrm{approx}_{f-p}^{(s)}(A) \|_2.
 \end{align*}
 Using a result by Crouzeix and Palencia~\cite{CrouzeixPalencia2017}, we have $\|(f-p)(A)\|_2 \le C \|f - p \|_{W(A)}$.
Since the spectral norm of a block diagonal matrix is the maximum spectral norm of its blocks, it holds that
\begin{align*}
 \| \mathrm{approx}_{f-p}^{(s)}(A) \|_2 & \le \max_{i} \|(f-p)(D_i)\|_2 + \max_{i} \|(f-p)(B_i)\|_2 + \max_{i,j} \|(f-p)(C_i^{(j)})\|_2 \\
 & \le 3 C\| (f-p) \|_{W(A)}.
\end{align*}
In the latter inequality, we used again~\cite{CrouzeixPalencia2017}  combined with the fact that the numerical range of a principal submatrix of $A$ is contained in $W(A)$. 
We conclude that
\begin{equation*}
 \| f(A) - \mathrm{approx}_f^{(s)}(A) \|_2 \le 4 C \| (f-p) \|_{W(A)},
\end{equation*}
and the claim follows from taking the minimum over all polynomials $p \in \Pi_m$.
\end{proof}

When considering the approximation of the \emph{trace} of a matrix function by Algorithm~\ref{alg:bandedpoly}, a stronger convergence result could be proved, because of the exactness of the low-rank updates (and therefore of the D\&C algorithm) for polynomials in $\Pi_{2m}$. In the specific case of Algorithm~\ref{alg:bandedpoly}, however, we can prove a stronger result even for the diagonal entries of $f(A)$.


\begin{theorem}\label{thm:diagBandedExact}
Let $A \in \R^{n \times n}$ with bandwidth $b$, let us fix a block size $s$, let $m := \lfloor \frac{s}{2b} \rfloor$. Then the output $\mathrm{approx}_p^{(s)}(A)$ of Algorithm~\ref{alg:bandedpoly} satisfies
 \begin{equation}\label{eq:diagexact}
  \diag(p(A))  = \diag( \mathrm{approx}_p^{(s)}(A)) 
  \end{equation}
  for all polynomials $p \in \Pi_{2m+1}$.
 \end{theorem}
 
 \begin{proof}
 The proof is in the spirit of~\cite[Lemma 5.1]{PozzaTudisco2018}, but the aim is different. 
 By linearity of Algorithm~\ref{alg:bandedpoly}, it is sufficient to prove~\eqref{eq:diagexact} when $p(x) = x^k$, with $0 \le k \le 2m + 1$, that is, to prove that the diagonal entries of $A^k$ and $\mathrm{approx}_p^{(s)}(A)$ coincide. To study the entries of $A^k$, it is helpful to consider the associated directed graph $\mathcal{G}(A)$ with vertices $1, \ldots, n$ and adjacency matrix $A$. The $j$th diagonal entry of $A^k$ is given by the sum of the weights of all the paths of length exactly $k$ that start and end at vertex $i$; we recall that the weight of a path $v_1 \to v_2 \to \ldots \to v_k$ of length $k$ is defined as the product of the weights of the edges $\prod_{h=1}^{k-1} A_{v_h v_{h+1}}$. We also consider the graphs $\mathcal{G}(B_i)$, $\mathcal{G}(D_i)$, $\mathcal{G}(C_i^{(1,2)})$. The diagonal entries of $\mathrm{approx}_p^{(s)}(A)$ are obtained by summing the weights of the paths of length exactly $k$ in the graphs $\mathcal{G}(D_i)$ and $\mathcal{G}(B_i)$ and subtracting the weights of the paths of length exactly $k$ in the graphs $\mathcal{G}(C_i^{(1)})$ and $\mathcal{G}(C_i^{(2)})$ for all indices $i$. Therefore, it is sufficient to prove that this sum coincides with the sum of the weights of the paths of length exactly $k$ in $\mathcal{G}(A)$.
 
 
 Note that, for all indices $i$, $\mathcal{G}(C_i^{(1)})$ is a subgraph of $\mathcal{G}(D_i)$ and $\mathcal{G}(B_i)$; $\mathcal{G}(C_i^{(2)})$ is a subgraph of $\mathcal{G}(D_{i+1})$ and $\mathcal{G}(B_i)$; all these are subgraphs of $\mathcal{G}(A)$. The distance from a vertex in $\mathcal{G}(D_i)$ and one in $\mathcal{G}({B}_{i+1})$ or $\mathcal{G}({B}_{i-2})$ is at least $m+1$. Therefore, for each vertex $v \in \{1, \ldots, n\}$ each path in $\mathcal{G}(A)$ of length at most $2m+1$ from $v$ to itself satisfies one (and only one) of the following conditions for some $i \in \{1, \ldots, \frac{n}{s}-1\}$:
 \begin{enumerate}
  \item It is contained in $\mathcal{G}({C}_i^{(1)})$, $\mathcal{G}({B}_i)$, and $\mathcal{G}({D}_i)$, but in no other subgraph.
  \item It is contained in  $\mathcal{G}({C}_i^{(2)})$, $\mathcal{G}({B}_i)$, and $\mathcal{G}({D}_{i+1})$, but in no other subgraph.
  \item It is contained in $\mathcal{G}({B}_i)$ but in no other subgraph.
  \item It is contained in $\mathcal{G}({D}_i)$ but in no other subgraph.
 \end{enumerate}
 In all these four cases, the weight of the path is counted exactly once in $\mathrm{approx}_p^{(s)}(A)$; we conclude that the diagonal entries of $\mathrm{approx}_p^{(s)}(A)$ coincide with the ones of $p(A)$ for $p(x) = x^k$ for $k \le 2m+1$ and therefore for all polynomials in $\Pi_{2m+1}$.
\end{proof}

 
 
 A convergence result for the diagonal elements of the output of Algorithm~\ref{alg:bandedpoly} follows from Theorem~\ref{thm:diagBandedExact} similarly to Theorem~\ref{thm:convpoly}.
 \begin{corollary}\label{cor:diagtrace}
  With the same assumptions of Theorem~\ref{thm:convpoly} it holds that 
 \begin{equation*}
  |f(A)_{ii} - \mathrm{approx}_f^{(s)}(A)_{ii} | \le 4C \min_{p \in \Pi_{2m+1}} \| f - p \|_{W(A)}
  \end{equation*}
  for all $i = 1, \ldots, n$ and therefore
  \begin{equation*}
   |\trace(f(A)) - \trace(\mathrm{approx}_f^{(s)}(A))| \le 4C n \min_{p \in \Pi_{2m+1}} \| f - p \|_{W(A)},
  \end{equation*}
  where $C = 1$ for normal matrices $A$, and $C = 1 + \sqrt{2}$ otherwise.
 \end{corollary}
 
 \begin{proof}
   According to Theorem~\ref{thm:diagBandedExact}, for all polynomials $p \in \Pi_{2m+1}$ we have that
  \begin{equation*}
   |f(A)_{ii} - \mathrm{approx}_f^{(s)}(A)_{ii} | = |(f-p)(A)_{ii} - \mathrm{approx}_{f-p}^{(s)}(A)_{ii} | \le \| (f-p)(A) - \mathrm{approx}_{f-p}^{(s)}(A) \|_2.
  \end{equation*}
  From here one proceeds as in the proof of Theorem~\ref{thm:convpoly}. \qedhere
%
%
%
 \end{proof}
 

 In Figure~\ref{fig:doublespeedbanded} we illustrate the convergence of $\mathrm{approx}_f^{(m)}(A)$ for the exponential of two banded matrices and we observe that the diagonal -- and therefore the trace -- converges much faster than the full matrix function.
 
 \begin{figure}[ht]
     \centering
     \begin{subfigure}[b]{0.4\textwidth}
         \centering
         \includegraphics[width=\textwidth]{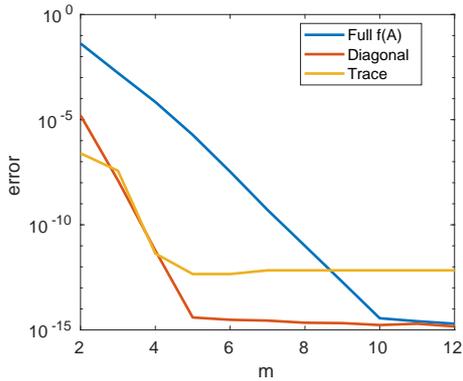}
         \caption{Normalized random symmetric tridiagonal matrix.}
     \end{subfigure}
     \hfill
     \begin{subfigure}[b]{0.4\textwidth}
         \centering
         \includegraphics[width=\textwidth]{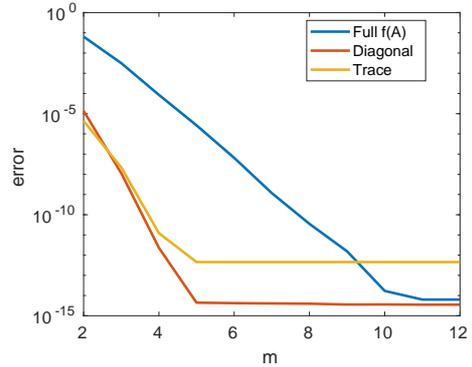}
         \caption{Normalized random non-symmetric pentadiagonal matrix.}
     \end{subfigure}
        \caption{Convergence of the errors $\|f(A) - \mathrm{approx}_f^{(m)}(A)\|_F$,
 $\|\text{diag}(f(A) - \mathrm{approx}_f^{(m)}(A) )\|_2$, and 
 $|\trace(f(A) - \mathrm{approx}_f^{(m)}(A) )|$ for $f=\exp$.}
 \label{fig:doublespeedbanded}
\end{figure}

\subsection{Adaptive algorithm}\label{sec:adaptive}
In Algorithm~\ref{alg:bandedpoly} the block size $s$, which determines the accuracy of the approximation of $f(A)$, needs to be chosen a priori and is uniform across the whole matrix. In the following, we develop a strategy to choose the block size adaptively and possibly differently in different parts of the matrix.

When $f$ is a polynomial of degree $m$ and $A$ has bandwidth $b$, $f(A)$ has bandwidth (at most) $bm$ and the discussion in Section~\ref{sec:convbanded} implies that Algorithm~\ref{alg:adaptiveSplitting} is exact for block size $s = 2bm$. This motivates the following strategy. For a target accuracy $\varepsilon$, we define the \emph{$\varepsilon$-approximate bandwidth} of a matrix to be the bandwidth that the matrix has if we discard all the entries with absolute value smaller than $\varepsilon$. In the first phase, we choose the sizes of the blocks $D_1, D_2, \ldots, D_k$ in such a way that their sizes are at least twice the $\varepsilon$-approximate bandwidth of $f(D_1), f(D_2), \ldots, f(D_k)$, and we set $F := \mathrm{blkdiag}(f(D_1), \ldots, f(D_k))$. In the second phase we compute the ``updates" between each pair of consecutive blocks $D_j$ and $D_{j+1}$ corresponding to indices $[j_1, h]$ and $[h+1, j_2]$ of $A$, respectively, similarly to~\eqref{eq:lala}. More precisely, we take 
\begin{equation}\label{eq:splittingupdate}
    P := f(B) - \mathrm{blkdiag}(f(C^{(1)}), f(C^{(2)})), \quad B := A(J, J), \quad C^{(1)} := A(J_1, J_1), \quad C^{(2)} := A(J_2, J_2)
\end{equation}
for $J_1 := \left [ \lfloor \frac{j_1 + h}{2} \rfloor, h \right ]$, $J_2 := \left [ h+1, \lceil \frac{j_2 + h}{2} \rceil \right ]$, and $J := J_1 \cup J_2$, and add the matrix $P$ to the submatrix of $F$ corresponding to the indices $J$. As a heuristic criterion to check convergence,
we check if the absolute value of all the entries corresponding to the first and last column and row of $P$ is smaller than $\varepsilon$; if this is not the case, the sets $J_1$, $J_2$, and $J$ are enlarged. The procedure is summarized in Algorithm~\ref{alg:adaptiveSplitting}.

\begin{algorithm}
\caption{Block diagonal splitting algorithm: Adaptive version}
\label{alg:adaptiveSplitting}
 \begin{algorithmic}[1]
  \REQUIRE{Banded matrix $A \in \R^{n \times n}$, tolerance $\varepsilon$, function $f$, minimum block size $n_{\min}$}
  \ENSURE{Approximation $F$ of $f(A)$}
  \STATE{Initialize $F \leftarrow \texttt{zeros}(n)$, $s \leftarrow n_{\min}$, $i \leftarrow 1$ ($i$ denotes where the next diagonal block starts)}
  \WHILE{$ i \le n $}
    \IF{ $f(A(\texttt{i:i+s-1, i:i+s-1}))$ has $\varepsilon$-approximate bandwidth $\le s/2$}
        \STATE{Set $F(\texttt{i:i+s-1, i:i+s-1}) = f(A(\texttt{i:i+s-1, i:i+s-1}))$, $i \leftarrow i+s$, $s \leftarrow \min\{s/2, n_{\min}\}$}
    \ELSE 
        \STATE{Choose a larger block size $s \leftarrow \min\{2s, n-i+1\}$}
    \ENDIF
\ENDWHILE
\FOR{each pair of consecutive diagonal blocks}
    \STATE{Compute $P$ using matrices $B$, $C^{(1)}$, $C^{(2)}$ corresp. to indices $J$, $J_1$, and $J_2$ as in~\eqref{eq:splittingupdate}}
    \WHILE{the update has not converged}
    \STATE{Enlarge matrices $B$, $C^{(1)}$, $C^{(2)}$ in~\eqref{eq:splittingupdate} corresp. to indices $J$, $J_1$, and $J_2$, and recompute $P$}
    \ENDWHILE
    \STATE{Sum $F(J,J) \leftarrow F(J,J) + P$}
\ENDFOR
 \end{algorithmic}
\end{algorithm}

\section{Numerical tests for Algorithm~\ref{alg:adaptiveSplitting}}\label{sec:examplepoly}

\subsection{Fermi-Dirac density matrix of one-dimensional Anderson model}\label{sec:fermidirac}
As a first numerical experiment, we test Algorithm~\ref{alg:adaptiveSplitting} on the function $f(z) = \left ( \exp(\beta(z-\mu))+1\right )^{-1}$ and a symmetric tridiagonal matrix with diagonal entries uniformly randomly distributed in $[0, 1]$ and all other nonzero elements equal to $-1$, as in~\cite[Section 5]{BenziRazouk2007}; this is the Fermi–Dirac density matrix corresponding to a one-dimensional Anderson model. We use $\mu = 0.5$ and $\beta = 1.84$. We set $\varepsilon = 10^{-5}$, $n_{\min} = 32$, and we consider values of $n$ ranging from $2^9$ to $2^{19}$. For each value of $n$, we compare the approximation $F$ returned by Algorithm~\ref{alg:adaptiveSplitting} to the approximation $p(A)$ where $p$ is a Chebyshev polynomial interpolating $f$ on $[-2, 3]$ of degree $d := \lceil \texttt{nnz}(F) / (2n) \rceil$; choosing the degree in this way gives a banded approximation of $f(A)$ with roughly the same storage cost and a comparable accuracy. The results are reported in Table~\ref{tab:fermidirac}; the approximation errors (relative errors in the Frobenius norm) and the timings are comparable.

\begin{table}[ht]
 	\centering
 			\begin{filecontents*}{data/testfermidirac.dat}
512	0.022117	4.4196e-07	47	0.027006	0.014933	1.5879e-07
1024	0.029615	4.5609e-07	47.5	0.11574	0.018006	1.5918e-07
2048	0.036756	4.5643e-07	47.75	0.59886	0.024464	1.6093e-07
4096	0.055749	4.5815e-07	47.875	3.344	0.048128	1.609e-07
8192	0.15626	4.5925e-07	47.938	21.505	0.11765	1.6097e-07
16384	0.36186	4.5967e-07	47.969	150.5	0.27184	1.6058e-07
32768	0.50538	0	47.984	0	0.58957	0
65536	1.0672	0	47.992	0	1.107	0
1.3107e+05	2.2349	0	47.996	0	2.667	0
2.6214e+05	4.6758	0	47.998	0	5.3822	0
5.2429e+05	9.2362	0	47.999	0	11.475	0
\end{filecontents*}
\pgfplotstabletypeset[
 		every head row/.style={
 			before row={
 				\toprule
 				{$A$} &\multicolumn{3}{c|}{Splitting algorithm}
 	&\multicolumn{2}{c|}{Chebyshev interpolation} & Dense \\
 			},
 			after row = \midrule,
 		},
 		every last row/.style={ after row=\bottomrule},
 		highlightcell/.style={
            postproc cell content/.append code={
              \ifnum\pgfplotstablerow=#1\relax%
                  \pgfkeysalso{@cell content/.add={$\bf}{$}}
              \fi
            }
        },
 		columns = {0,1,2,3,5,6,4},
 		columns/0/.style = {column name = Size, fixed, column type = c|},
 		columns/1/.style = {column name = Time, fixed,  highlightcell={6}, highlightcell={7}, highlightcell={8}, highlightcell={9}, highlightcell={10}, zerofill},
 		columns/2/.style = {column name = Err, skip rows between index={6}{11}, zerofill },
 		columns/3/.style = {column name = $\texttt{nnz}/n$, column type=c|, zerofill},
 		columns/4/.style = {column name = Time, skip rows between index={6}{11}, fixed, column type = c,  zerofill},
 		columns/5/.style = {column name = Time, fixed, highlightcell={0}, highlightcell={1}, highlightcell={2}, highlightcell={3}, highlightcell={4}, highlightcell={5}, zerofill},
 		columns/6/.style = {column name = Err, skip rows between index={6}{11}, column type = c|, zerofill},
 		]{data/testfermidirac.dat}
 	\caption{Computation of $f(A)$ by Algorithm~\ref{alg:adaptiveSplitting}, where $f(z) = \left ( \exp(\beta(z-\mu))+1\right )^{-1}$ and the matrices $A$ are symmetric tridiagonal matrices with diagonal entries uniformly randomly distributed in $[0, 1]$ and all other nonzero elements equal to $-1$, as discussed in Section~\ref{sec:fermidirac}. 
 	}
 	\label{tab:fermidirac}
 \end{table}
 
 \subsection{Spectral adaptivity: Comparison with interpolation by Chebyshev polynomials}\label{sec:cheby}
 
 An advantage of (polynomial) Krylov subspace over polynomial interpolation on the spectral interval of $A$ is the fact that Krylov methods are less impacted by outliers in the spectrum of $A$. In the next experiment, we consider three $2048 \times 2048$ matrices:
 \begin{itemize}
 \item The exponential of $A_1 = \mathrm{tridiag}[-1, 2, -1]$;
 \item The exponential of the matrix $A_2$ which is obtained from $A_1$ by changing the entry in position $(1,1)$ to $10$;
 \item The square root of the matrix $A_3$ which is the tridiagonal matrix with \texttt{linspace(2, 3, n)} on the diagonal and $-1$ on the first super- and sub-diagonals.
\end{itemize}

We run Algorithm~\ref{alg:bandedpoly} with different block sizes and Chebyshev interpolation with different degrees of Chebyshev polynomial and we plot in Figure~\ref{fig:cheby} the relative error in the Frobenius norm versus the number of nonzero entries in the approximation of the matrix functions described above. For the matrix $A_1$, Chebyshev outperforms Algorithm~\ref{alg:bandedpoly}. However, for the matrix $A_2$ which has an outlier in the eigenvalues, and for the matrix $A_3$ for which it is difficult to find a good polynomial approximation on the whole spectral interval, Algorithm~\ref{alg:bandedpoly} achieves a smaller error with the same number of nonzero entries.

\begin{figure}[htb]
\centering
 \includegraphics[scale=0.55]{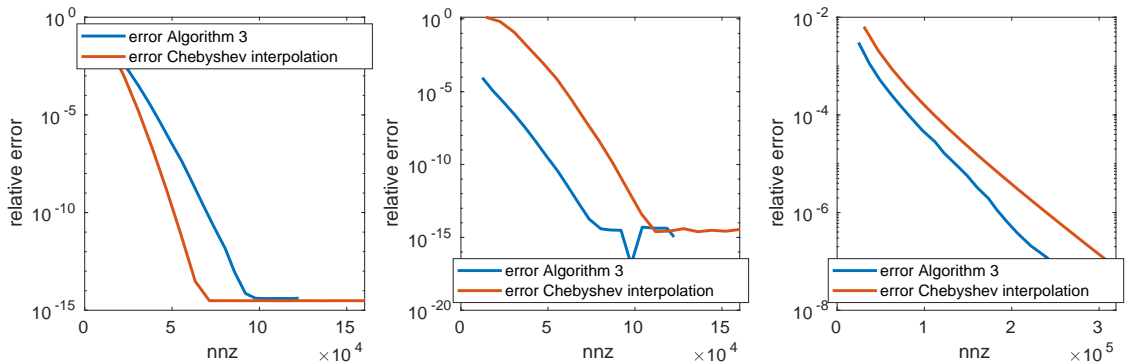}
 \caption{Relative error in the Frobenius norm of the approximations of $\exp(A_1)$,  $\exp(A_2)$, and $\sqrt{A_3}$ from Section~\ref{sec:cheby} obtained by Algorithm~\ref{alg:bandedpoly} and by Chebyshev interpolation.}
 \label{fig:cheby}
\end{figure}

We do not report the timings: In general, Chebyshev interpolation is faster than our splitting algorithm; however, Chebyshev interpolation is only suitable for symmetric matrices (for non-symmetric matrices one needs more refined techniques such as using Faber polynomials as discussed, e.g., in~\cite{BenziRazouk2007}), while the splitting method works for any banded matrix, can automatically adapt to different spectral distributions, and could exploit the Toeplitz structure of $A$ producing an approximation in constant time (as the matrices $D$, $B$, and $C$ are made of equal blocks, we could compute only a constant number of matrix functions of the small blocks).

\subsection{Adaptivity in the size of blocks}\label{sec:adaptiveblocks}
The matrix square root of $A_3$ has slower off-diagonal decay in the upper-left region, as shown in Figure~\ref{fig:adaptiveSpy}(b). We run Algorithm~\ref{alg:adaptiveSplitting} to compute $\sqrt{A_3}$, setting $\varepsilon = 10^{-8}$. The sparsity pattern of the output is shown in Figure~\ref{fig:adaptiveSpy}(a), where different block sizes are selected for different parts of the matrix; the relative error of the computed approximation  is $2.6 \cdot 10^{-10}$ in the Frobenius norm.

 \begin{figure}[htb]
     \centering
     \begin{subfigure}[b]{0.39\textwidth}
         \centering
         \includegraphics[width=0.63\textwidth]{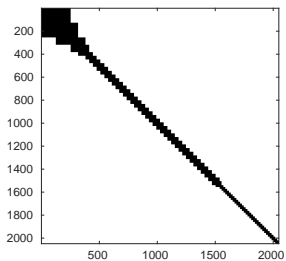}
         \caption{Sparsity structure of the output of Algorithm~\ref{alg:adaptiveSplitting} applied to $A_3$ and $f = \sqrt{~}$.}
     \end{subfigure}
     \hfill
     \begin{subfigure}[b]{0.39\textwidth}
         \centering
         \includegraphics[width=0.73\textwidth]{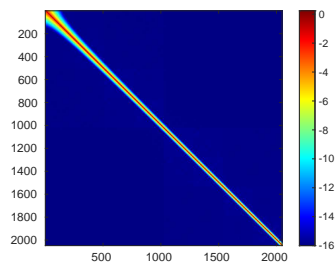}
         \caption{Logarithm of absolute values of entries of $\sqrt{A_3}$.}
     \end{subfigure}
        \caption{The matrix $A_3$ is the tridiagonal matrix with \texttt{linspace(2, 3, n)} on the diagonal and $-1$ on the first super- and sub-diagonals.}
 \label{fig:adaptiveSpy}
\end{figure}


\subsection{Comparison with HSS algorithm}\label{sec:times}

\begin{wrapfigure}[15]{r}[0pt]{0.45\textwidth}
\vspace{-15pt}
 \includegraphics[scale=.5]{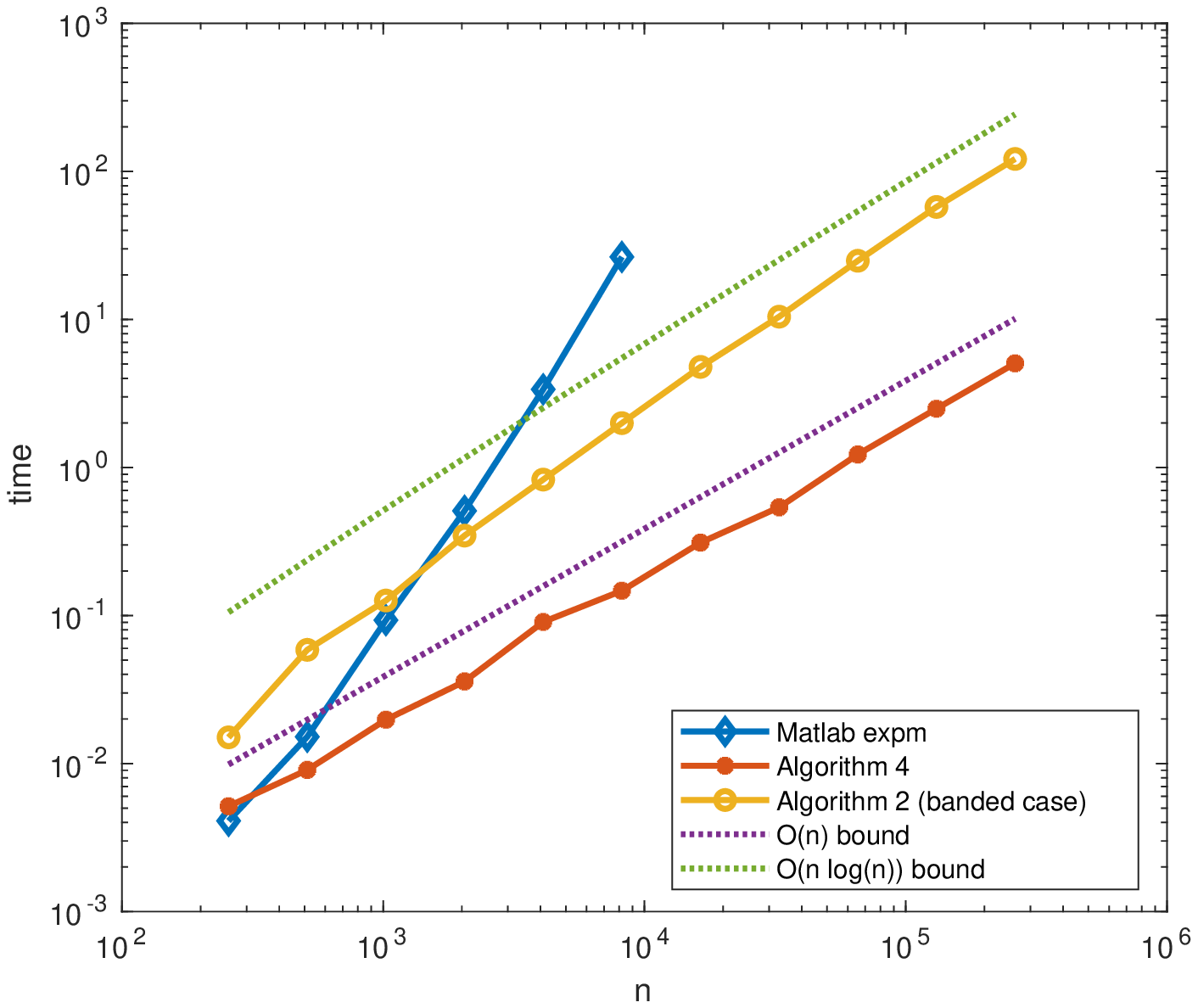}
\caption{Timings of Algorithms~\ref{alg:adaptiveSplitting} and~\ref{alg:D&C} for $\exp(-\mathrm{tridiag}(-1, 2, -1))$.}
\label{fig:times}
\end{wrapfigure}

We expect Algorithm~\ref{alg:adaptiveSplitting} to be faster than the general D\&C algorithm (Algorithm~\ref{alg:D&C}) as the first one should scale as $\mathcal O(n)$ and the latter as $\mathcal O(n \log n)$, plus the fact that we have no overhead computations needed for HSS arithmetic. We compare the timings of the two algorithms for the computation of $\exp(-A)$ for $A = \mathrm{tridiag}(-1, 2, -1)$. For Algorithm~\ref{alg:adaptiveSplitting} we use a minimum block size of $64$, while for Algorithm~\ref{alg:D&C} we set $n_{\min}=128$ and we write each low-rank update as a rank-$1$ update as discussed in Remark~\ref{rmk:spdrankb}; in both cases we set the tolerance parameter $\varepsilon=10^{-8}$. We report the results in Figure~\ref{fig:times}, together with the timings of Matlab's \texttt{expm}, for matrix dimensions ranging from $n = 2^8$ to $n = 2^{18}$.

\section{Conclusions}\label{sec:conclusions}

In this work we have proposed two new algorithms for computing matrix functions of structured matrices, based on a D\&C paradigm. The algorithms have been tested on a wide range of examples of practical relevance that require to compute, for a medium- to large-scale matrix,  the whole matrix function, its diagonal or its trace. The numerical results demonstrate that, most of the time, the proposed methods outperform state-of-art techniques with respect to time consumption and offer a comparable accuracy. 
For the convergence analysis of these algorithms, we have also expanded the framework of low-rank updates of matrix functions \cite{BeckermannKressnerSchweitzer2018,BCKS2020} towards several directions. In the Hermitian case, we have shown that the approximation of the trace of the update, computed by projection on the polynomial Krylov subspace, has a higher convergence rate with respect to the full update. For the splitting algorithm, we have provided a convergence analysis that highlights stronger convergence properties for the entries located on the main diagonal, which applies also to non-Hermitian matrix arguments. 

\paragraph{Acknowledgments.} 

The authors would like to thank Bernhard Beckermann, Paola Boito, and Stefan Güttel for helpful discussions on topics related to this work.
	
	\bibliographystyle{abbrv}
	\bibliography{Bib} 
	
	\end{document}